\documentclass[11pt]{amsart}

\usepackage{amscd, amsfonts, amsmath, amssymb, amsthm, enumerate, epsfig, flexisym, float, graphicx, hhline, hyperref, mathtools, pdfpages, pinlabel, subcaption, tabu, tikz, tikz-cd, url, wasysym}
\usepackage[autostyle]{csquotes}
\usepackage[all,cmtip]{xy}

\makeatletter
\@namedef{subjclassname@2020}{Mathematics Subject Classification \textup{2020}}
\makeatother

\newtheorem{theorem}{Theorem}[section]
\newtheorem{corollary}[theorem]{Corollary}
\newtheorem{lemma}[theorem]{Lemma}
\newtheorem{proposition}[theorem]{Proposition}

\theoremstyle{definition}

\newtheorem{remark}[theorem]{Remark}
\numberwithin{equation}{subsection}
\newtheorem*{ack}{Acknowledgement}

\newcommand{\Conj}{\operatorname{Conj}}
\newcommand{\Core}{\operatorname{Core}}
\newcommand{\Alex}{\operatorname{Alex}}

\newcommand{\rk}{\operatorname{rk}}

\newcommand{\Env}{\operatorname{Env}}

\newcommand{\Inn}{\operatorname{Inn}}

\newcommand{\Artin}{\operatorname{\mathcal{A}}}

\newcommand{\Z}{\operatorname{Z}}

\newcommand{\Aut}{\operatorname{Aut}}

\newcommand{\mcg}{\mathcal{M}}
\newcommand{\dq}{\mathcal{D}}

\usetikzlibrary{cd, arrows, calc, knots, patterns, positioning, shapes.callouts, shapes.geometric, shapes.symbols}
\tikzset{>=stealth', arrow/.style={->}}

\begin{document}

\title{Dehn quandles of groups and orientable surfaces}
\author{Neeraj K. Dhanwani}
\author{Hitesh R. Raundal}
\author{Mahender Singh}

\address{Department of Mathematical Sciences, Indian Institute of Science Education and Research (IISER) Mohali, Sector 81, S. A. S. Nagar, P. O. Manauli, Punjab 140306, India.}
\email{neerajk.dhanwani@gmail.com}
\email{hiteshrndl@gmail.com}
\email{mahender@iisermohali.ac.in}


\begin{abstract}
Unifying various constructions of quandles including Coxeter quandles, free quandles, knot quandles of prime knots and Dehn quandles of orientable surfaces, we introduce Dehn quandles of groups with respect to their subsets. It turns out that Dehn quandles are precisely the ones that embed naturally into their enveloping groups. We prove that the enveloping group of the Dehn quandle of a given group with respect to its generating set is a central extension of that group, and that enveloping groups of Dehn quandles of Artin groups and link groups with respect to their standard generating sets are the groups themselves.  We discuss orderability of Dehn quandles and prove that free involutory quandles are left orderable, whereas certain generalised Alexander quandles are bi-orderable.  Specialising to surfaces, we give generating sets for Dehn quandles of orientable surfaces with punctures and compute their automorphism groups. As applications, we recover a result of Niebrzydowski and Przytycki proving that the knot quandle of the trefoil knot is isomorphic to the Dehn quandle of the torus and also extend a result of Yetter on epimorphisms of Dehn quandles of orientable surfaces onto certain involutory homological quandles.
\end{abstract}

\maketitle

\section{Introduction}
Quandles are algebraic systems with a binary operation that encodes the three Reidemeister moves of planar diagrams of links in the 3-space. These objects have shown appearance in a wide spectrum of mathematics including knot theory \cite{Joyce1979, Joyce1982, Matveev1982}, group theory, mapping class groups \cite{Zablow1999, Zablow2003}, set-theoretic solutions of the quantum Yang-Baxter equation and Yetter-Drinfeld Modules \cite{Eisermann2005}, Riemannian symmetric spaces \cite{Loos1} and Hopf algebras \cite{Andruskiewitsch2003}, to name a few. Though quandles already appeared under different guises in the literature, their study gained true momentum after the fundamental works of Joyce \cite{Joyce1982} and Matveev \cite{Matveev1982}, who showed that link quandles are complete invariants of non-split links up to orientation of the ambient space. Although quandles are strong invariants of links, the isomorphism problem for them is hard. This has motivated search for newer properties, constructions and invariants of quandles themselves.
\par

Groups are natural sources of quandles. Unifying various known constructions of quandles including Coxeter quandles, free quandles and Dehn quandles of orientable surfaces, we introduce Dehn quandles of groups with respect to their subsets. Recall that a free quandle on a set is simply a union of conjugacy classes of generators in the free group on that set equipped with the quandle operation of conjugation. Let $\mathcal{M}_g$ be the mapping class group of a closed orientable surface $S_g$ and $\mathcal{D}_g$ the set of isotopy classes of simple closed curves in $S_g$. It is well-known that $\mathcal{M}_g$ is generated by Dehn twists along finitely many elements from $\mathcal{D}_g$ \cite[Theorem 4.1]{Farb-Margalit2012}. The binary operation
$\alpha * \beta= T_\beta(\alpha),$ where $\alpha, \beta \in \mathcal{D}_g$ and $T_\beta$ is the Dehn twist along $\beta$, turns $\mathcal{D}_g$ into a quandle called the Dehn quandle of the surface $S_g$. Identifying the isotopy class of a simple closed curve with the corresponding Dehn twist, it turns out that the quandle $\mathcal{D}_g$ is a subquandle of the conjugation quandle $\Conj(\mathcal{M}_g)$ of the mapping class group. These quandles originally appeared in the work of Zablow \cite{Zablow1999, Zablow2003}. Further, \cite{kamadamatsumoto,Yetter2002, Yetter2003} considered a quandle structure on the set of isotopy classes of simple closed arcs in orientable surfaces with at least two punctures, and called it quandle of cords. In the case of a disk with $n+1$ punctures, we will see that this quandle is simply the Dehn quandle of the braid group $B_{n+1}$ with respect to its standard set of generators. A presentation for the quandle of cords of the plane and the 2-sphere has been given in \cite{kamadamatsumoto}. More generally, in the follow-up \cite{DhanwaniRaundalSingh2022} of this work, we have given two approaches to write explicit presentations for the class of Dehn quandles.
\par

Zablow derived some fundamental relations in Dehn quandles \cite{Zablow2008} and subsequently used them to develop a homology theory based on these quandles. It is shown that isomorphism classes of Lefschetz fibrations over a disk correspond to quandle homology classes in dimension two. In \cite{ChamanaraHuZablow}, the Dehn quandle structure of the torus has been extended to a quandle structure on the set of its measured geodesic foliations and the quandle homology of this extended quandle has been studied. Niebrzydowski and Przytycki \cite{NiebrzydowskiPrzytycki2009} proved that the Dehn quandle of the torus is isomorphic to the fundamental quandle of the trefoil knot. It was suspected that there may be close connections between knot quandles and Dehn quandles. However, the trefoil knot turned out to be somewhat special and the result does not extend in a large number of cases. For instance, it has been proved in \cite{Zablow2014} that if $K$ is a torus knot of type $(2,p)$, where $p$ is odd and $3\nmid p$, then $K$ does not admit any non-trivial coloring by elements of a Dehn quandle $\mathcal{D}_g$ for any orientable surface $S_g$ of genus $g \ge 1$. At the same time, all torus knots of type $(3,n)$, where $n$ is even, admit non-trivial Dehn quandle colorings.
\par

In this paper, we introduce Dehn quandles of groups with respect to their subsets. These quandles include many well-known constructions of quandles from groups including the conjugation quandle of a group, the free quandle on a set, Coxeter quandles, Dehn quandles of closed orientable surfaces, quandle of cords of orientable surfaces, knot quandles of prime knots, core quandles of groups and generalised Alexander quandles of groups with respect to fixed-point free automorphisms, to name a few. Our construction places many well-known results in the subject into a unified perspective.
\par

The paper is broadly divided into two parts. The first part develops the general theory of Dehn quandles of groups and the second part focuses on Dehn quandles of orientable surfaces. The precise organisation of the paper is as follows. In Section \ref{section-prelim}, we recall some basic terminology and set notations.  We begin Section \ref{section-dehn-quandles} by making some basic observations on Dehn quandles of groups. We prove that the enveloping group of the Dehn quandle of a given group with respect to its generating set is a central extension of that group (Theorem \ref{central-extension}). As a characterisation, we prove that a quandle is a Dehn quandle of a group with respect to a generating set if and only if the natural map from the quandle to its enveloping group is injective (Proposition \ref{injectivity-dehn-eta}). Consequently, it follows that a knot is prime if and only if its knot quandle is the Dehn quandle of its knot group with respect to a single Wirtinger generator (Corollary \ref{knot-prime-dehn-quandle}). We prove that the enveloping group of the Dehn quandle of a group with respect to the generating set of a presentation that has only conjugation relations is the group itself (Theorem \ref{ker-phi-conjugacy-relation-trivial}). In Section \ref{section orderable dehn quandles}, we discuss orderability of Dehn quandles (Proposition \ref{dehn-quandle-not-ordererable}) and deduce that spherical Artin groups are not bi-orderable (Corollary \ref{artin not bi-order}). We also prove that free involutory quandles are left orderable (Proposition \ref{free invol left orderable}) and that certain generalised Alexander quandles are bi-orderable (Proposition \ref{biordered_alexander_quandle}). Section \ref{section Dehn quandles of surfaces} deals with Dehn quandles of orientable surfaces possibly with punctures, and determine generating sets for these quandles (Proposition \ref{cor:gen_Dehnquandle-with-punctures}). As an application of our construction, we recover a result of Niebrzydowski and Przytycki that the knot quandle of the trefoil knot is isomorphic to the Dehn quandle of the torus (Theorem \ref {Niebrzydowski Przytycki theorem}). In Section \ref{section auto dehn quandles surfaces}, we compute automorphism groups of Dehn quandles of orientable surfaces (Theorem \ref{Aut-Dehn-quandle-puncture}). In Section \ref{sec-homological-quandles}, we prove that there exist surjective quandle homomorphisms from Dehn quandles of surfaces of positive genus onto certain finite homological quandles arising from the algebraic intersection number (Theorem \ref{thm:homologicalquandle}). As a consequence, we generalise a similar result of Yetter \cite{Yetter2003} who considered the genus two case. Finally, we also prove that these homological quandles are simply Dehn quandles of corresponding symplectic groups (Proposition \ref{homological-are-dehn-of-symplectic}).
\medskip

\section{Preliminaries}\label{section-prelim}
We begin by recalling some basic definitions.
\par
A {\it quandle} is a non-empty set $Q$ together with a binary operation $*$ satisfying the following axioms:
\begin{enumerate}[(i)]
\item $x*x=x$\, for all $x\in Q$.
\item For each $x, y \in Q$, there exists a unique $z \in Q$ such that $x=z*y$.
\item $(x*y)*z=(x*z)*(y*z)$\, for all $x,y,z\in Q$.
\end{enumerate}
\par

The second axiom is equivalent to the bijectivity of the right multiplication by each element of $Q$. This gives a dual binary operation $*^{-1}$ on $Q$ defined as $x*^{-1}y=z$ if $x=z*y$. Thus, the second axiom is equivalent to saying that
\begin{equation*}
(x*y)*^{-1}y=x\qquad\textrm{and}\qquad\left(x*^{-1}y\right)*y=x
\end{equation*}
for all $x,y\in Q$, and hence it allows cancellations from right.
\par

Topologically, the three quandle axioms correspond to the three Reidemeister moves of planar diagrams of links in the 3-space. Following are some basic examples of quandles, some of which we shall use in the forthcoming sections.
\begin{itemize}
\item If $G$ is a group, then the binary operation $x*y=y x y^{-1}$ turns $G$ into the quandle $\Conj(G)$ called the \textit{conjugation quandle} of $G$.
\item A group $G$ with the binary operation $x*y=yx^{-1}y$ turns $G$ into the quandle $\Core(G)$ called the \textit{core quandle} of $G$. In particular, if $G$ is a cyclic group of order $n$, then $\Core(G$) is the \textit{dihedral quandle} of order $n$.
\item If $G$ is a group and $\phi\in \Aut(G)$, then $G$ with the binary operation $x*y=\phi\left(xy^{-1}\right)y$ forms a quandle $\Alex(G,\phi)$ referred as the \textit{generalized Alexander quandle} of $G$ with respect to $\phi$.
\item Every link can be assigned a quandle called the link quandle which is a complete invariant of non-split links up to weak equivalence. This fundamental result appeared independently in the works of Joyce \cite{Joyce1979, Joyce1982} and Matveev \cite{Matveev1982}.
\end{itemize}

Morphisms and automorphisms of quandles are defined in the obvious way. We denote the group of all automorphisms of a quandle $Q$ by $\Aut(Q)$. Note that the quandle axioms are equivalent to saying that for each $y\in Q$, the map $S_y:Q\to Q$ given by $S_y(x)=x* y$ is an automorphism of $Q$ fixing $y$. The group $\Inn(Q)$ generated by such automorphisms is called the group of inner automorphisms of $Q$. The group $\Inn(Q)$ acts on the quandle $Q$ and the corresponding orbits are referred as connected components of the quandle. Thus, a {\it connected} quandle is one that has only one orbit, that is, $\Inn(Q)$ acts transitively on $Q$. For example, a dihedral quandle of odd order is connected whereas that of even order is not. Further, $Q$ is called {\it involutory} if $x*^{-1}y=x* y$ for all $x,y\in Q$. For example, the core quandle of any group is involutory.
\medskip

The \textit{enveloping group} $\Env(Q)$ of a quandle $Q$ is the group with the set of generators as $\{e_x \mid x \in Q\}$ and the defining relations as
\begin{equation*}
e_{x*y}=e_y e_x e_y^{-1}
\end{equation*}
for all $x,y\in Q$. For example, if $Q$ is a trivial quandle, then $\Env(Q)$ is the free abelian group of rank equal to the cardinality of $Q$. The enveloping group of the link quandle $Q(L)$ of a link $L$ is the link group $G(L)$ of $L$ \cite{Joyce1979, Joyce1982}. The natural map
\begin{equation*}
\eta: Q \to \Env(Q)
\end{equation*}
given by $\eta(x)=e_x$ is a quandle homomorphism with $\Env(Q)$ viewed as the conjugation quandle. The map $\eta$ is not injective in general. The functor from the category of quandles to that of groups assigning the enveloping group to a quandle is left adjoint to the functor from the category of groups to that of quandles assigning the conjugation quandle to a group. Thus, enveloping groups play a crucial role in understanding of quandles themselves. 
\par

We conclude this section by setting some notation. By \cite[Lemma 4.4.7]{Winker1984}, any element in a quandle can be written in a left-associated product of the form
\begin{equation*}
\left(\left(\cdots\left(\left(a_0*^{\epsilon_1}a_1\right)*^{\epsilon_2}a_2\right)*^{\epsilon_3}\cdots\right)*^{\epsilon_{n-1}}a_{n-1}\right)*^{\epsilon_n}a_n,
\end{equation*}
which, for simplicity, we write as
\begin{equation*}
a_0*^{\epsilon_1}a_1*^{\epsilon_2}\cdots*^{\epsilon_n}a_n.
\end{equation*}
For elements $g, h$ of a group $G$, we denote the commutator $g hg^{-1}h^{-1}$ by $[ g, h ]$ and the element $h g h^{-1}$ by $g^h$.
\medskip

\section{Dehn quandles of groups}\label{section-dehn-quandles}
In this main section, we introduce and develop the theory of Dehn quandles of groups.

\subsection{Generators and closure properties of Dehn quandles of groups} 

Let $G$ be a group, $A$ a non-empty subset of $G$ and $A^G$ the set of all conjugates of elements of $A$ in $G$. The {\it Dehn quandle} $\mathcal{D}(A^G)$ of $G$ with respect to $A$ is defined as the set $A^G$ equipped with the binary operation of conjugation, that is, $$x*y=yxy^{-1}$$ for all $x, y \in \mathcal{D}(A^G)$.
\par

Clearly, $\mathcal{D}(A^G)$ is a subquandle of $\Conj(G)$ for each subset $A$ of $G$. We will see later that every subquandle of $\Conj(G)$ is a Dehn quandle, and that Dehn quandles of groups with respect to a set of generators are of particular interest. The terminology is justified since Dehn quandles are simply unions of conjugacy classes of elements in groups, and Dehn was the first to highlight the study of conjugacy classes in groups.
\par

Dehn quandles of groups generalise many well-known constructions of quandles from groups. First, notice that $\mathcal{D}(G^G)$ is the conjugation quandle $\Conj(G)$. If $F(S)$ is the free group generated by $S$, then $\mathcal{D}(S^{F(S)})$ is the free quandle on $S$ \cite{Kamada2012, Kamada2017, NosakaBook}. If $\mathcal{W}$ is a Coxeter group with Coxeter generating set $S$, then $\mathcal{D}(S^{\mathcal{W}})$ is the so called Coxeter quandle \cite{Nosaka2017, TAkita}. Notice that Coxeter quandles turn out to be involutory. Let $S_{g}$ be a closed orientable surface of genus $g$ and $\mathcal{M}_{g}$ its mapping class group. If $S$ is the set of Dehn twists about essential simple closed curves, then $\mathcal{D}(S^{\mathcal{M}_{g}})$ is the Dehn quandle of the surface \cite{Zablow1999, Zablow2003, kamadamatsumoto}. We will discuss Dehn quandles of surfaces in detail in the upcoming sections. 

\begin{remark}
We note that \cite{BardakovNasybullov2020} gives a similar but different construction of quandles from groups. By \cite[Proposition 4.3]{BardakovNasybullov2020}, the construction in \cite{BardakovNasybullov2020} agrees with the Dehn quandle $\mathcal{D}(A^G)$ if and only if elements of $A$ are pairwise non-conjugate in $G$. Thus, for a fairly large class of groups including knot groups, mapping class groups, Artin groups and Coxeter groups, our construction differs from the one in \cite{BardakovNasybullov2020}.
\end{remark}

We begin with the following basic and useful observation. 

\begin{proposition}\label{generators-dehn-quandle-general}
If $G$ is a group generated by $S$, then $\mathcal{D}(S^G)$ is generated as a quandle by $S$.
\end{proposition}

\begin{proof}
If $x \in \mathcal{D}(S^G)$, then $x=gsg^{-1}$ for some $s \in S$ and $g \in G$. Since $S$ generates $G$ as a group, we have $g=s_1^{\epsilon_1}s_2^{\epsilon_2} \cdots s_k^{\epsilon_k}$ for some $s_i \in S$ and $\epsilon_i \in \{1, -1\}$. Thus, we can write $x=s_1^{\epsilon_1}s_2^{\epsilon_2} \cdots s_k^{\epsilon_k}s\, s_k^{-\epsilon_k}\cdots s_2^{-\epsilon_2} s_1^{-\epsilon_1} = s *^{\epsilon_k} s_k *^{\epsilon_{k-1}} s_{k-1} *^{\epsilon_{k-2}} \cdots *^{\epsilon_1} s_1.$
\end{proof}

We now observe some closure properties of Dehn quandles.

\begin{proposition}\label{prod_Dehn_quandle}
Let $A$ and $B$ be subsets of groups $G$ and $H$, respectively. Then $\dq(A^G) \times \dq(B^H) \cong \dq( (A \times B)^{G\times H})$.
\end{proposition}

\begin{proof}
The proof follows from the fact that conjugacy classes in direct products of two groups correspond to products of conjugacy classes.
\end{proof}
 
Recall that the disjoint union of two quandles can be viewed as a quandle where elements of one quandle acts trivially on elements of the other. The following observation is a direct consequence of definitions.

\begin{proposition}\label{prop:sum_Dehn_quandle} 
Let $A$ and $B$ be subsets of groups $G$ and $H$, respectively. Then $\dq(A^G)\sqcup \dq(B^H) \cong \dq((A\sqcup B)^{G\times H})$, where $A$ and $B$ are viewed as subsets of $G \times H$ via natural inclusions.
\end{proposition}

Given two quandle $Q_1=\langle S_1\mid R_1\rangle$ and $Q_2= \langle S_2\mid R_2\rangle$, we can define their free product as $$Q_1 \star Q_2=\langle S_1\sqcup S_2 \mid R_1\sqcup R_2\rangle.$$
 
\begin{proposition}\label{prop:free_prod_Dehn_quandle} 
Let $A$ and $B$ be subsets of groups $G$ and $H$, respectively. Then $\dq(A^G) \star \dq(B^H)$ is a Dehn quandle.
\end{proposition}

\begin{proof}
The proof follows from \cite[Lemma 7.1]{BardakovNasybullov2020} and Proposition \ref{injectivity-dehn-eta}.
\end{proof}
 
It is routine to check that propositions  \ref{prod_Dehn_quandle}, \ref{prop:sum_Dehn_quandle} and  \ref{prop:free_prod_Dehn_quandle} hold for arbitrary direct products, unions and free products, respectively.

\subsection{Enveloping groups of Dehn quandles of groups}
For a group $G$ and a subset $A$ of $G$, let $c(A^G)$ denote the number of connected components of $\mathcal{D}(A^G)$. Note that, $c(A^G)$ equals the number of conjugacy classes of elements of $G$ represented by elements of $A$.

\begin{theorem}\label{gen-group-dehn-quandle}
Let $G$ be a group generated by $S$. Then the following hold:
\begin{enumerate}[(i)]
\item $\Env(\mathcal{D}(S^G))$ is generated by $\{e_s \mid s\in S \}$.
\item $\Env(\mathcal{D}(S^G))_{ab}\cong \mathbb{Z}^{c(S^G)}$.
\end{enumerate}
\end{theorem}

\begin{proof}
Any element $x \in S^G$ can be written in the form $x=s_1^{\epsilon_1}s_2^{\epsilon_2} \cdots s_k^{\epsilon_k} s_0 s_k^{-\epsilon_k}\cdots s_2^{-\epsilon_2} s_1^{-\epsilon_1}$ for some $s_i \in S$ and $\epsilon_i \in \{1, -1\}$. Then, we see that
\begin{align*}
e_x &= e_{s_1^{\epsilon_1}s_2^{\epsilon_2} \cdots s_k^{\epsilon_k} s_0 s_k^{-\epsilon_k}\cdots s_2^{-\epsilon_2} s_1^{-\epsilon_1}}\\
&= e_{s_0 *^{\epsilon_k} s_k *^{\epsilon_{k-1}} s_{k-1}*^{\epsilon_{k-2}}\cdots *^{\epsilon_1} s_1}\\ 
&= e_{s_1}^{\epsilon_1} e_{s_2}^{\epsilon_2} \cdots e_{s_k}^{\epsilon_k} e_{s_0} e_{s_k}^{-\epsilon_k} \cdots e_{s_2}^{-\epsilon_2} e_{s_1}^{-\epsilon_1},
\end{align*}
which proves assertion (i).
\par
In view of (i), the abelianization $\Env(\mathcal{D}(S^G))_{ab}$ is generated by cosets $[e_s]$ for $s \in S$. Furthermore, these cosets satisfy relations $[e_{gsg^{-1}}]=[e_s]$ and $[e_s][e_{s'}]=[e_{s'}][e_s]$ for $s, s' \in S$ and $g \in G$. Thus, $[e_s]=[e_{s'}]$ in $\Env(\mathcal{D}(S^G))_{ab}$ if and only if $s, s' \in S$ are conjugate in $G$. Hence, $\Env(\mathcal{D}(S^G))_{ab}$ is a free abelian group of rank $c(S^G)$, which is assertion (ii).
 \end{proof}
 
We note that the second assertion in Theorem \ref{gen-group-dehn-quandle} also follows from \cite[Proposition 3.3(1)]{BardakovNasybullovSingh2019}. Let $G$ be a group generated by $S$ and $$\Phi: \Env(\mathcal{D}(S^G)) \to G$$ be the map defined by $\Phi(e_x)= x$ for $x \in \mathcal{D}(S^G)$. Since $\Phi(e_{x*y})=\Phi(e_{yxy^{-1}})=yxy^{-1}=\Phi(e_ye_xe_y^{-1})$ for all $x, y \in \mathcal{D}(S^G)$
and $S$ generates $G$, it follows that $\Phi$ is a surjective group homomorphism.

\begin{theorem}\label{central-extension}
Let $G$ be a group generated by $S$. Then 
\begin{equation}\label{exact-sequence-dehn-group}
1 \to \ker(\Phi) \to \Env(\mathcal{D}(S^G)) \to G \to 1
\end{equation}
is a central extension of groups. Further, if $G_{ab}$ has torsion, then the extension \eqref{exact-sequence-dehn-group} does not split.
\end{theorem}

\begin{proof}
By Theorem \ref{gen-group-dehn-quandle}(i), $\Env(\mathcal{D}(S^G))$ is generated by $\{e_s \mid s\in S \}$. Thus, any $g \in \ker(\Phi)$ can be written as $g=e_{s_1}^{\epsilon_1}e_{s_2}^{\epsilon_2} \cdots e_{s_k}^{\epsilon_k}$. This gives $s_1^{\epsilon_1}s_2^{\epsilon_2} \cdots s_k^{\epsilon_k}=\Phi(e_{s_1}^{\epsilon_1}e_{s_2}^{\epsilon_2} \cdots e_{s_k}^{\epsilon_k})=\Phi(g)=1$. Now, for any $e_x \in \Env(\mathcal{D}(S^G))$, we have
\begin{align*}
g e_x g^{-1}&=e_{s_1}^{\epsilon_1}e_{s_2}^{\epsilon_2} \cdots e_{s_k}^{\epsilon_k} e_x e_{s_k}^{-\epsilon_k}\cdots e_{s_2}^{-\epsilon_2} e_{s_1}^{-\epsilon_1}&\\
&=e_{x*^{\epsilon_k}s_k*^{\epsilon_{k-1}}s_{k-1}*^{\epsilon_{k-2}}\cdots*^{\epsilon_1}s_1}&\left(\text{by relations in}\; \Env\!\left(\dq\!\left(S^G\right)\right)\right)\\
&=e_{s_1^{\epsilon_1}s_2^{\epsilon_2}\cdots s_k^{\epsilon_k} x\;\!s_k^{-\epsilon_k}\cdots s_2^{-\epsilon_2}s_1^{-\epsilon_1}}&\\
&=e_x&\text{(since $s_1^{\epsilon_1}s_2^{\epsilon_2} \cdots s_k^{\epsilon_k}=1$)}.
\end{align*}
Hence, the extension \eqref{exact-sequence-dehn-group} is central. Now, suppose that $G_{ab}$ has torsion and \eqref{exact-sequence-dehn-group} splits. Then we have $\Env(\mathcal{D}(S^G)) \cong \ker(\Phi) \times G $. Taking abelianization and using Theorem \ref{gen-group-dehn-quandle}(ii), we get $\mathbb{Z}^{c(S^G)} \cong \Env(\mathcal{D}(S^G))_{ab} \cong \ker(\Phi) \times G_{ab}$, a contradiction. Hence, the sequence does not split.
\end{proof}

The next result gives characterisations of Dehn quandles.

\begin{proposition}\label{injectivity-dehn-eta}
The following statements are equivalent for any quandle $Q$:
\begin{enumerate}[(i)]
\item $Q$ embeds in $\Conj(H)$ for some group $H$.
\item The natural map $\eta: Q \to \Env(Q)$ is injective.
\item $Q \cong \mathcal{D}(S^G)$ for some group $G$ and a generating set $S$ of $G$.
\end{enumerate}
\end{proposition}

\begin{proof}
For (i) $\Rightarrow$ (ii), suppose that $\iota: Q \hookrightarrow \Conj(H)$ is an embedding for some group $H$. Then, by \cite[p.42]{Joyce1982}, there is a unique group homomorphism $\hat{\iota}: \Env(Q) \to H$ such that $\iota=\hat{\iota} ~\eta$. Thus, $\eta$ is injective. The implication (ii) $\Rightarrow$ (i) is obvious.
\par
For (ii) $\Rightarrow$ (iii), suppose that $Q$ is a quandle for which the map $\eta$ is injective. Let $A$ be the set of representatives of orbits (connected components) of $Q$. We claim that elements of $\eta(A)$ are pairwise non-conjugate in $\Env(Q)$. Let $x$ and $y$ be distinct elements of  $A$ such that $e_x=g e_y g^{-1}$, where $g= e_{x_1}^{\epsilon_1} e_{x_2}^{\epsilon_2} \cdots e_{x_k}^{\epsilon_k}$ for $x_i \in Q$ and $\epsilon_i \in \{ 1, -1\}$. Then, 
$$e_x=e_{x_1}^{\epsilon_1} e_{x_2}^{\epsilon_2} \cdots e_{x_k}^{\epsilon_k} e_y e_{x_k}^{-\epsilon_k} \cdots e_{x_2}^{-\epsilon_2} e_{x_1}^{-\epsilon_1} = e_{y *^{\epsilon_k} x_k *^{\epsilon_{k-1}} x_{k-1} *^{\epsilon_{k-2}} \cdots *^{\epsilon_1} x_1}.$$
Since $\eta$ is injective, it follows that $x=y *^{\epsilon_k} x_k *^{\epsilon_{k-1}} x_{k-1}*^{\epsilon_{k-2}} \cdots *^{\epsilon_1} x_1$, that is, $x$ and $y$ are in the same orbit, a contradiction. It now follows from \cite[Proposition 4.3 and Theorem 4.5]{BardakovNasybullov2020} that $Q\cong \mathcal{D}(\eta(A)^{\Env(Q)})$. If $S$ is a generating set for $Q$, then $\eta(S)$ is a generating set for $\Env(Q)$. Further, $S$ intersects every orbit of $Q$, and therefore we can choose a representative of an orbit from $S$. In other words, the set $A$ can be chosen to be a subset of $S$. For this choice of $A$, elements of $\eta(S)$ are conjugates of elements of $\eta(A)$ in $\Env(Q)$. Thus, $\eta(S)^{\Env(Q)}=\eta(A)^{\Env(Q)}$ as sets, and hence $Q\cong \dq(\eta(S)^{\Env(Q)})$.
\par
For (iii) $\Rightarrow$ (ii), suppose that $Q \cong \mathcal{D}(S^G)$ for some group $G$ and a generating set $S$ of $G$. Let $\tau: \mathcal{D}(S^G) \to \Conj(G)$ be the natural embedding. Further, we have quandle homomorphisms $\eta: \mathcal{D}(S^G) \to \Conj(\Env(\mathcal{D}(S^G)))$ and $\Phi: \Conj(\Env(\mathcal{D}(S^G))) \to \Conj(G)$. Since $\Phi ~\eta=\tau$, it follows that $\eta$ is injective.
\end{proof}

In fact, if $Q$ is a subquandle of $\Conj(G)$ for some group $G$, then $Q \cong \dq(Q^{\langle Q \rangle })$, where $\langle Q \rangle$ is the subgroup of $G$ generated by $Q$. Following are some immediate consequences of Proposition \ref{injectivity-dehn-eta}.

\begin{corollary}\label{subquandle_of_conj_is_dehn_quandle}
If $G$ is a group, then every subquandle of $\Conj(G)$ is a Dehn quandle.
\end{corollary}

\begin{corollary}\label{knot-prime-dehn-quandle}
A knot $K$ is prime if and only if its knot quandle $Q(K)$ is isomorphic to $\mathcal{D}(x^{G(K)})$, where $x$ is a Wirtinger generator of the knot group $G(K)$.
\end{corollary}

\begin{proof}
It is known from \cite{Ryder1996} that $K$ is a prime knot if and only if the map $\eta: Q(K) \to G(K)$ is injective. Since the knot quandle $Q(K)$ is known to be connected \cite[Corollary 15.3]{Joyce1982}, the assertion now follows from Proposition \ref{injectivity-dehn-eta}.
\end{proof}

As an application of Corollary \ref{knot-prime-dehn-quandle}, we shall retrieve the main result of \cite{NiebrzydowskiPrzytycki2009} in Theorem \ref{Niebrzydowski Przytycki theorem}.

\begin{proposition}\label{core-are-dehn}
The core quandle of a group is a Dehn quandle of some group.
\end{proposition}

\begin{proof}
If $G$ is a group, then it follows from \cite[Proposition 6.2]{Bergman2021} that $\Core(G)$ embeds in $\Conj(H)$ for some group $H$. The result now follows from Proposition \ref{injectivity-dehn-eta}.
\end{proof}

 An automorphism of a group is {\it fixed-point free} if it does not fix any non-identity element of the group.

\begin{proposition}\label{alexander_quandle_is_dehn_quandle}
Let $G$ be a group and $\phi$ a fixed-point free automorphism of $G$. Then the generalised Alexander quandle $\Alex(G,\phi)$ is a Dehn quandle.
\end{proposition}

\begin{proof}
We claim that the natural quandle homomorphism $\Alex(G,\phi)\to\Conj(\Inn(\Alex(G,\phi)))$ defined by $x\mapsto S_x$ is injective. Suppose that $S_x=S_y$ for $x,y\in\Alex(G,\phi)$. Then, we have 
$S_x(z)=S_y(z)$ for all $z \in \Alex(G,\phi)$. This give $\phi(z)\phi(x)^{-1}x=\phi(z)\phi(y)^{-1}y$, which further implies that $\phi(xy^{-1})=xy^{-1}$. Since $\phi$ is fixed-point free,  we must have $x=y$.
Thus, $\Alex(G,\phi)$ embeds in $\Conj(\Inn(\Alex(G,\phi)))$, and the result now follows from Corollary \ref{subquandle_of_conj_is_dehn_quandle}.
\end{proof}
\medskip

Recall that a discrete group is said to be {\it rationally acyclic} if its homology groups with coefficients in the trivial module $\mathbb{Q}$ vanish in all dimensions more than zero. Examples of rationally acyclic groups include Coxeter groups \cite[Theorem 15.1.1]{Davis2008}, Higman's group, binate groups \cite{Berrick1989} and symmetric groups on infinite sets \cite{HarpeMcDuff}. Let $\rk(G)$ denote the torsion-free rank of an abelian group $G$, that is, $\rk(G) = \dim_{\mathbb{Q}} (G \otimes_{\mathbb{Z}} \mathbb{Q})$. The following result generalises \cite[Proposition 4.5]{TAkita}.

\begin{theorem}
If $G$ is a rationally acyclic group generated by $S$, then $\rk (\ker(\Phi))=c(S^G)$. Further, there is an isomorphism of cohomology rings 
\begin{equation*}
\mathrm{H}^*(\Env(\mathcal{D}(S^G)), \mathbb{Q}) \cong \mathrm{H}^*(\ker(\Phi), \mathbb{Q}).
\end{equation*}
\end{theorem}

\begin{proof}
By \cite[Chapter VII, Corollary 6.4]{Brown1982}, the central extension 
\begin{equation}\label{central-extension-equation}
1 \to \ker(\Phi) \to \Env(\mathcal{D}(S^G)) \to G \to 1
\end{equation}
 gives the exact sequence 
$$\mathrm{H}_2(G,\mathbb{Q}) \to \mathrm{H}_1(\ker(\Phi) ,\mathbb{Q}) \to \mathrm{H}_1( \Env(\mathcal{D}(S^G)), \mathbb{Q}) \to \mathrm{H}_1(G,\mathbb{Q}) \to 0$$ of rational homology groups. Since $G$ is rationally acyclic, the preceding exact sequence gives
$\mathrm{H}_1(\ker(\Phi) ,\mathbb{Q}) \cong \mathrm{H}_1( \Env(\mathcal{D}(S^G)), \mathbb{Q})$. Thus, we have

\begin{align*}
\rk (\ker(\Phi))&=\dim_{\mathbb{Q}}\!\left(\ker(\Phi) \otimes_{\mathbb{Z}} \mathbb{Q}\right)&\\
&=\dim_{\mathbb{Q}} \mathrm{H}_1\!\left(\ker(\Phi), \mathbb{Q}\right)&\\
&=\dim_{\mathbb{Q}} \mathrm{H}_1\!\left(\Env\!\left(\mathcal{D}\!\left(S^G\right)\right), \mathbb{Q}\right)&\\
&=\dim_{\mathbb{Q}}\!\left(\Env\!\left(\mathcal{D}\!\left(S^G\right)\right)_{ab} \otimes_{\mathbb{Z}} \mathbb{Q}\right)&\\
&=c(S^G)&\text{(by Theorem \ref{gen-group-dehn-quandle})}.
\end{align*}
The second assertion follows by applying the Lyndon-Hochschild-Serre spectral sequence $E_r^{*,*}$ to the central extension \eqref{central-extension-equation}. Notice that the second page terms of the spectral sequence are given by $E_2^{p, q}= \mathrm{H}^p (G, H^q(\ker(\Phi),\mathbb{Q}))$ and the spectral sequence converges to $\mathrm{H}^{p+q}(\Env(\mathcal{D}(S^G)), \mathbb{Q})$. Since $G$ is rationally acyclic, the universal coefficient theorem gives $\mathrm{H}^p(G,\mathbb{Q}) = 0$ for $p \ge 1$. This together with the fact that the extension is central implies that $E_2^{p, q}= \mathrm{H}^p (G, \mathbb{Q}) \otimes \mathrm{H}^q(\ker(\Phi),\mathbb{Q})= \mathrm{H}^q(\ker(\Phi),\mathbb{Q})$ for $p =0$ and 0 for $p \ge 1$.
\end{proof}

\begin{theorem}\label{ker-phi-conjugacy-relation-trivial}
If $G$ is a group with a presentation $\langle S \mid R\rangle$ such that $R$ consists of relations only of the form $w s w^{-1}= t$ for some $s, t \in S$ and $w \in G$, then $$\Env(\mathcal{D}(S^G))\cong G.$$
\end{theorem}

\begin{proof}
We show that the surjective group homomorphism $\Phi: \Env(\mathcal{D}(S^G)) \to G$ is, in fact, an isomorphism. Any relation in the presentation $\langle S \mid R\rangle$ of $G$ is of the form $w s w^{-1}= t$ for some $s, t \in S$ and $w \in G$. We write $w=s_1^{\epsilon_1} s_2^{\epsilon_2} \cdots s_k^{\epsilon_k}$ for some $s_i \in S$ and $\epsilon_i \in \{1, -1\}$. Define a homomorphism $\psi: F(S) \to \Env(\mathcal{D}(S^G))$ by setting $\psi(s)= e_s$ for each $s\in S$, where $F(S)$ is the free group on $S$. By a repeated use of relations of the form $e_{y^\epsilon xy^{-\epsilon}} = e_y^\epsilon e_x e_y^{-\epsilon}$ in $\Env(\mathcal{D}(S^G))$, we see that
\begin{align*}
\psi(w s w^{-1}) &= \psi\!\left(s_1^{\epsilon_1} s_2^{\epsilon_2} \cdots s_k^{\epsilon_k} s s_k^{-\epsilon_k} \cdots s_2^{-\epsilon_2} s_1^{-\epsilon_1}\right)\\
&= e_{s_1}^{\epsilon_1} e_{s_2}^{\epsilon_2} \cdots e_{s_k}^{\epsilon_k} e_s e_{s_k}^{-\epsilon_k} \cdots e_{s_2}^{-\epsilon_2} e_{s_1}^{-\epsilon_1}\\
&= e_{s_1^{\epsilon_1} s_2^{\epsilon_2} \cdots s_k^{\epsilon_k} s s_k^{-\epsilon_k} \cdots s_2^{-\epsilon_2} s_1^{-\epsilon_1}}\\
&= e_{w s w^{-1}}\\
&= e_t\\
&=\psi(t).
\end{align*}
Thus, we have a homomorphism $\widetilde{\psi}: G \to \Env(\mathcal{D}(S^G))$. By Theorem \ref{gen-group-dehn-quandle}(i), $ \Env(\mathcal{D}(S^G))$ is generated by $\{e_s \mid s \in S \}$, and hence $\widetilde{\psi}$ is surjective. Finally, since $\widetilde{\psi}\Phi$ is the identity map, it follows that $\Phi$ is an isomorphism of groups.
\end{proof}

Recall that an Artin group $\Artin$ is a group with a presentation $$\Artin=\{s_1,s_2,\ldots,s_n~\mid~ (s_is_j)_{m_{ij}}=(s_js_i)_{m_{ij}}, m_{ij}\in \{2,3,4,\ldots\} \cup \{ \infty\}\},$$ where $(s_is_j)_{m_{ij}}$ is the word $s_is_js_is_j s_i \ldots$ of length $m_{ij}$ if $m_{ij} < \infty$ and there is no relation $(s_is_j)_{m_{ij}}=(s_js_i)_{m_{ij}}$ if $m_{ij}=\infty$. We set $S= \{ s_1, s_2,\ldots, s_n \}$ and refer to $m_{ij}$'s as exponents. Notice that $\Artin_{ab}\cong \mathbb{Z}^{c(S^\Artin)}$. A Coxeter group $\mathcal{W}$ is a quotient of $\Artin$ by imposing additional relations $s_i^2=1$ for all $s_i \in S$. The pair $(\mathcal{W}, S)$ is referred as a Coxeter system. For brevity, Dehn quandles of Artin and Coxeter groups with respect to their standard generating sets will be referred as Artin and Coxeter quandles, respectively.

\begin{corollary}\label{ker-phi-artin-trivial}
If $\Artin$ is an Artin group generated by $S$, then $\Env(\mathcal{D}(S^\Artin))\cong\Artin$.
\end{corollary}

\begin{proof}
Notice that a relation $(s_is_j)_{m_{ij}}=(s_js_i)_{m_{ij}}$ in $\Artin$ can be written in terms of conjugation as
\begin{eqnarray*}
(s_is_j)_{m_{ij}-1} ~s_j~ ((s_is_j)_{m_{ij}-1})^{-1}=s_j & if & m_{ij}~\textrm{is even}, \label{artin-relation-1}\\
(s_is_j)_{m_{ij}-1} ~s_i~ ((s_is_j)_{m_{ij}-1})^{-1}=s_j & if & m_{ij}~\textrm{is odd}.
\end{eqnarray*}
The result now follows from Theorem \ref{ker-phi-conjugacy-relation-trivial}.
\end{proof}

It is known from \cite{TAkita} that if $(\mathcal{W}, S)$ is a Coxeter system, then $\Env(\mathcal{D}(S^{\mathcal{W}}))$ is an intermediate group between $\mathcal{W}$ and the corresponding Artin group. Thus, the preceding corollary shows a contrast between Dehn quandles of Artin groups and that of Coxeter groups.
\par
Let $L$ be an oriented link in the 3-sphere, $D(L)$ its link diagram and $S=\{x_1, x_2, \ldots, x_n\}$ its set of labelled arcs. Then the link group $G(L)= \pi_1(\mathbb{S}^3 \setminus L)$ of $L$ is generated by $S$ and has defining relations at each crossing in $D(L)$ as shown in Figure \ref{fig1}. Thus, all relations in the Wirtinger presentation of $G(L)$ are conjugation relations.

\begin{figure}[hbt!]
\begin{subfigure}{0.4\textwidth}
\centering
\begin{tikzpicture}[scale=0.6]
\node at (0.4,-1.2) {{\small $x$}};
\node at (-1.2,0.4) {{\small $y$}};
\node at (1.1,1.2) {{\small $yxy^{-1}$}};
\begin{knot}[clip width=6, clip radius=4pt]
\strand[->] (-2,0)--(2,0);
\strand[->] (0,-2)--(0,2);
\end{knot}
\end{tikzpicture}
\caption{Positive crossing}
\end{subfigure}
\begin{subfigure}{0.4\textwidth}
\centering
\begin{tikzpicture}[scale=0.6]
\node at (1.1,1.2) {{\small $y^{-1}x y$}};
\node at (-1.2,0.4) {{\small $y$}};
\node at (0.4,-1.2) {{\small $x$}};
\begin{knot}[clip width=6, clip radius=4pt]
\strand[->] (2,0)--(-2,0);
\strand[->] (0,-2)--(0,2);
\end{knot}
\end{tikzpicture}
\caption{Negative crossing}
\end{subfigure}
\caption{Relations at a positive and at a negative crossing}
\label{fig1}
\end{figure}
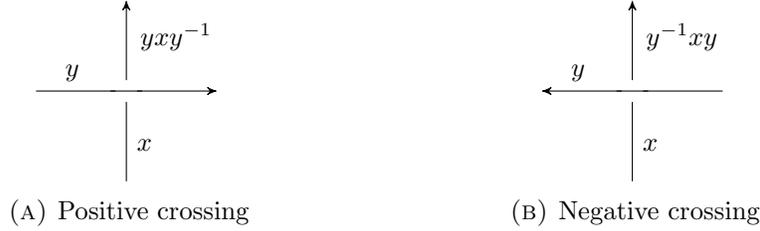

\begin{corollary}\label{ker-phi-link-group-trivial}
If $L$ is a link and $S$ a Wirtinger generating set for $G(L)$, then $$\Env(\dq(S^{G(L)})) \cong G(L).$$
\end{corollary}
\medskip

Let $ct(S^G)$ denote the number of conjugacy classes of elements of $G$ represented by torsion elements of $S$. Note that, if every element of $S$ is torsion, then $ct(S^G)=c(S^G)$.

\begin{proposition}\label{ker-phi-contain-ct}
If $G$ is a group generated by $S$, then $\mathbb{Z}^{ct(S^G)} \le \ker(\Phi)$. Further, the equality holds for Artin groups, Coxeter groups and link groups with their standard generating sets.
\end{proposition}

\begin{proof}
Let $s \in S$ be a torsion element, say $s^n=1$. Then $\Phi(e_s^n)=s^n=1$, and hence $e_s^n \in \ker(\Phi)$. Since $\ker(\Phi)$ is central in $\Env(\mathcal{D}(S^G))$, it follows that $e_{s}^n=e_x e_{s}^n e_x^{-1}= (e_xe_{s} e_x^{-1})^n= e_{x s x^{-1}}^n$ for all $x \in \mathcal{D}(S^G)$. This implies that $\mathbb{Z}^{ct(S^G)} \le \ker(\Phi)$. 
\par
Since Artin generators of Artin groups are torsion free, the equality holds by Corollary \ref{ker-phi-artin-trivial}. For Coxeter groups, the equality is proved in \cite[Theorem 3.1]{TAkita}. Finally, since Wirtinger generators of link groups are torsion free, the equality holds by Corollary \ref{ker-phi-link-group-trivial}.
\end{proof}
\medskip

We note that Proposition \ref{ker-phi-contain-ct} extends a recent result of Akita \cite[Theorem 3.1]{TAkita} who considered enveloping groups of Coxeter quandles. Next, we obtain a presentation for the enveloping group of Dehn quandle of a surface group of genus more than one. Recall that, for a closed orientable surface $S_g$, we have $$\pi_1(S_g)=\langle a_1,b_1,\ldots a_g,b_g \mid \prod_{i=1}^{g}[a_i,b_i]=1 \rangle.$$ Let $S= \{ a_1,b_1,\ldots a_g,b_g\}$. 

\begin{theorem}
If $g \ge 2$, then
$$\Env(\dq(S^{\pi_1(S_g)})) \cong \langle t,a_1,b_1,\ldots, a_g,b_g \mid [t,a_i] =[t,b_i]=1, \quad \prod_{i=1}^{g}[a_i,b_i]=t \rangle.$$
\end{theorem}

\begin{proof}
By Theorem \ref{central-extension}, we have the central extension $$1\to \ker(\Phi) \to E \stackrel{\Phi}{\to} \pi_1(S_g)\to 1,$$
where $E:=\Env(\mathcal{D}(S^{\pi_1(S_g)}))$. By \cite[Chapter VII, Corollary 6.4]{Brown1982}, we get the exact sequence 
\begin{equation}\label{cohomology-seq-surface-group}
\cdots \to \mathrm{H}_2(\pi_1(S_g), \mathbb{Z})\to \mathrm{H}_1(\ker(\Phi), \mathbb{Z})\to \mathrm{H}_1(E, \mathbb{Z})\mapsto \mathrm{H}_1(\pi_1(S_g), \mathbb{Z})\to 0.
\end{equation}
 Note that $\mathrm{H}_1(\pi_1(S_g), \mathbb{Z}) \cong \mathbb{Z}^{2g}$ and $\mathrm{H}_2(\pi_1(S_g), \mathbb{Z}) \cong \mathbb{Z}$. Further, by Theorem \ref{gen-group-dehn-quandle}(ii), we have $\mathrm{H}_1(E, \mathbb{Z}) \cong \mathbb{Z}^{2g}$. Thus, the exact sequence \eqref{cohomology-seq-surface-group} takes the form $$\cdots\to \mathbb{Z}\to \ker(\Phi) \to \mathbb{Z}^{2g}\to \mathbb{Z}^{2g}\to 0.$$ Since every surjective homomorphism from a free abelian group to itself is an isomorphism, we deduce that $\ker(\Phi)$ is a cyclic group. Now, we have the following cases.
\begin{enumerate}[(1)]
\item $\ker(\Phi)$ is trivial: In this case, $E\cong \pi_1(S_g)$. Since each relation in $E$ can be written as a conjugation relation between the generators, it follows that the relation $\prod_{i=1}^{2g}[a_i,b_i]=1$ can also be written as $w x w^{-1}=y$ for some $x, y \in S$ and $w\in \pi_1(S_g)$. If $x \neq y$, then rank of $\mathrm{H}_1(E, \mathbb{Z})$ is less than $2g$, a contradiction. If $x=y$ and $w$ is non-trivial, then we have a free abelian group of rank two $\langle x, w\rangle$ inside $\pi_1(S_g)$, again a contradiction, since $\pi_1(S_g)$ is a Fuchsian group. Hence, this case does not arise.
\item  $\ker(\Phi)=\langle t \mid t^n=1\rangle$ for $n>1$: By \cite[Lemma 2.5]{Labruere-Paris2001}, the group $E$ has a presentation
$$E\cong \langle t, a_1,b_1,\ldots, a_g,b_g \mid t^n=[t,a_i]=[t,b_i]=1,\quad \prod_{i=1}^{g}[a_i,b_i]=t^k \rangle,$$
for some integer $k$ with $1 \le k \le n$. If $\gcd(k, n) \neq 1$, then the abelianization of $E$ will have a torsion element of order $\gcd(k, n)$, which is a contradiction. Thus, we can assume that $\gcd(k, n) =1$. We claim that in this case the relation $t^n=1$ cannot be written as a conjugation relation. Suppose that we can do so, that is, there exist $x, y \in S$ and non-trivial $w\in E$ such that $w x w^{-1}=y$. If $x, y$ are distinct, then $\mathrm{H}_1(E, \mathbb{Z})$ has rank less then $2g$, which is contradiction. If $x=y$ and $w$ is non-trivial, then the group $\langle x,w\rangle$ is abelian, and hence $\Phi(\langle x,w\rangle)$ is an abelian subgroup of $\pi_1(S_g)$. Thus, $\Phi(\langle x,w\rangle)$ must be an infinite cyclic group, and $\Phi(w)$ is a power of $\Phi(x)$. This implies that $w\in\langle t,x\rangle$, and hence the relation $t^n=1$ can be recovered from the relation $[t,x]=1$. Thus, the relation $t^n=1$ can be removed from the presentation of $E$, in which case $t$ will be of infinite order, a contradiction. Hence, this case does not arise.
\item  $\ker(\Phi)= \langle t \rangle$: Again, by \cite[Lemma 2.5]{Labruere-Paris2001}, a presentation of $E$ is 
$$E\cong \langle t,a_1,b_1,\ldots, a_g,b_g \mid [t,a_i]=[t,b_i]=1,\quad \prod_{i=1}^{g}[a_i,b_i]=t^k\rangle$$ 
for some integer $k$ with $|k| \ge0$. The case $k=0$ is not possible as it gives rank of $\mathrm{H}_1(E, \mathbb{Z})$ to be $2g+1$, a contradiction. Further, $|k| \ge 2$ gives torsion in $\mathrm{H}_1(E, \mathbb{Z})$, again a contradiction. Hence, the only possibility is 
$$E \cong \langle t,a_1,b_1,\ldots, a_g,b_g \mid [t,a_i]=[t,b_i]=1,\quad \prod_{i=1}^{g}[a_i,b_i]=t^\epsilon \rangle,$$ 
where $\epsilon= \pm 1$. But, the groups for both the choices of $\epsilon$ are isomorphic, which is desired.
\end{enumerate}
\end{proof}
\par

\subsection{Automorphisms of Dehn quandles of groups}
We conclude this section with some observations on automorphisms of Dehn quandles. The following result \cite[Theorem 3.1]{Nosaka2017} is motivated by the construction of an augmented quandle \cite{Joyce1979}.

\begin{theorem}\label{action-on-quandle} 
Let $G$ be a group admitting a right action on a quandle $X$ and $\kappa: X \to G$ is a map such that $\kappa(X)$ generates $G$. 
\begin{enumerate}[(i)]
\item If $x*y=x \cdot \kappa(y)$ for all $x, y \in X$, then the action induces a surjective group homomorphism  $G \to \Inn(X)$.
\item If the action is effective (faithful), then $G \cong \Inn(X)$ and the action of $G$ on $X$ agrees with the natural $\Inn(X)$ action on $X$.
\end{enumerate}
\end{theorem}

As a consequence, we obtain.

\begin{proposition}\label{action-corollary}
Let $G$ be a group and $X$ be a subquandle of $\Conj(G)$ such that $X$ generates $G$. Then $\Inn(X) \cong G/\Z(G) \cong \Inn(\Conj(G))$. 
\end{proposition}

\begin{proof}
Since $X$ generates the group $G$, each element $g \in G$ can be written as $g=x_1^{\epsilon_1}x_2^{\epsilon_2}\cdots x_n^{\epsilon_n}$ for $x_i \in X$ and $\epsilon_i= \pm 1$. Further, since $X$ is a subquandle of $\Conj(G)$, the usual conjugation action of $G$ on itself keeps $X$ invariant, and hence induces an action on $X$. Let $\kappa: X \to G$ be the inclusion map. By Theorem \ref{action-on-quandle}(i), there is a surjective group homomorphism $f: G \to \Inn(X)$. But, $\ker(f)=\Z(G)$, and hence $G/\Z(G) \cong \Inn(X)$. On the other hand, for any group $G$, we have $\Inn(\Conj(G))\cong G/\Z(G)$.
\end{proof}

\begin{corollary}\label{cor-single-conj-class}
Let $G$ be a group generated by the conjugacy class $X$ of an element of $G$. Then the following hold:
\begin{enumerate}[(i)]
\item $\Inn(\Conj(G))\cong G/\Z(G) \cong \Inn(X)$. 
\item $\Aut_X(G) \le \Aut (X)$, where $\Aut_X(G)$ consists of group automorphisms of $G$ that preserve $X$ and $\Aut(X)$ is the group of quandle automorphisms of $X$.
\item $X$ is a connected quandle.
\end{enumerate}
\end{corollary}

\begin{proof}
Assertion (i) follows from Proposition \ref{action-corollary} once we notice that $X$ is a subquandle of $\Conj(G)$. Assertion (ii) follows since any $f \in \Aut_X(G)$ restricts to a quandle automorphism of $X$. Note that the conjugation action of $G$ and hence that of $G/\Z(G)$ on $X$ is transitive. Hence, by (i), $\Inn(X)$ acts transitively on $X$.
\end{proof}

In view of Proposition \ref{action-corollary}, we have 
\begin{corollary}
If $G$ is a group generated by $S$, then $\Inn(\mathcal{D}(S^G))\cong G/\Z(G).$
\end{corollary}

\begin{remark}
Computing $\Aut(\mathcal{D}(A^G))$ seems challenging in general. Recall that the Dehn quandle of a free group with respect to its free generating set is the free quandle on that set. A presentation of the automorphism group of the free quandle of rank $n$ is known \cite{FennRourke1997}, where it is shown to be isomorphic to the welded braid group on $n$ strands.
\end{remark}
\medskip

\section{Orderability of Dehn quandles of groups}\label{section orderable dehn quandles}
A quandle $Q$ is said to be \textit{left-orderable} if there is a (strict) linear order $<$ on $Q$ such that $x<y$ implies $z*x<z*y$ for all $x,y,z\in Q$. Similarly, a quandle $Q$ is \textit{right-orderable} if there is a (strict) linear order $<^\prime$ on $Q$ such that $x<^\prime y$ implies $x*z<^\prime y*z$ for all $x,y,z\in Q$. A quandle is \textit{bi-orderable} if it has a (strict) linear order with respect to which it is both left and right ordered. Orderability of groups is defined analogously. Orderability of quandles, particularly of link quandles, has been considered in detail in a recent work \cite{RaundalSinghSingh2020}. For Dehn quandles of groups, we have the following result.

 \begin{proposition}\label{dehn-quandle-not-ordererable}
Let $G$ be a group and $A$ its subset containing two distinct elements $x, y$ such that $xyx=yxy$. Then the Dehn quandle $\mathcal{D}(A^G)$ is neither right nor left orderable.
\end{proposition}

\begin{proof}
Notice that the braid relation $xyx=yxy$ can be written in $\mathcal{D}(A^G)$ in the form $x*y*x=y$ and $y*x*y=x$. Assume that there exists a right order $<$ on $\mathcal{D}(A^G)$. Without loss of generality, we can assume that $x<y$. By right orderability, we have $ x<y*x \implies x*y<y*x*y \implies x*y<x\implies x*y*x<x \implies y<x$, which is a contradiction. Similarly, suppose that we have a left order $<$ such that $x<y$. Then $y*x<y\implies x*y*x<x*y \implies y<x*y \implies y<y*x*y \implies y<x$, again a contradiction. Hence, $\mathcal{D}(A^G)$ is neither right nor left orderable.
\end{proof}

\begin{corollary}\label{group-with-braid-relation-not-ordererable}
Let $G$ be a group containing two distinct elements $x, y$ such that $xyx=yxy$. Then $G$ is not a bi-orderable group.
\end{corollary}

\begin{proof}
Let us set $A= \{x, y\}$. Then, by Proposition \ref{dehn-quandle-not-ordererable}, the quandle $\mathcal{D}(A^G)$ is not right orderable. By \cite[Proposition 3.4]{MR4450681}, if a group is bi-orderable, then its conjugation quandle is right orderable. Since $\mathcal{D}(A^G)$ is a subquandle of $\Conj(G)$, it follows that $G$ cannot be a bi-orderable group.
\end{proof}

The following corollary recovers the known result about failure of bi-orderability of spherical Artin groups \cite[Theorem 5.8.]{MulhollandRolfsen}, except for $I_2(2m)$ with $m \ge 3$.

\begin{corollary}\label{artin not bi-order} 
Artin groups with an odd exponent are not bi-orderable.
\end{corollary}

\begin{proof}
Let $s_i, s_j \in S$ be two Artin generators such that $m_{ij}$ is odd. Then, as in the proof of Proposition \ref{dehn-quandle-not-ordererable}, the relation $(s_is_j)_{m_{ij}}=(s_js_i)_{m_{ij}}$ can be written in $\mathcal{D}(S^\Artin)$ in the form $(s_i*s_j)_{m_{ij}}=s_j$ and $(s_j*s_i)_{m_{ij}}=s_i$. Now arguments as in the proof of Proposition \ref{dehn-quandle-not-ordererable} shows that $\mathcal{D}(S^\Artin)$ is not right orderable. By \cite[Proposition 3.4]{MR4450681}, if $\Artin$ is a bi-orderable group, then $\Conj(\Artin)$ is a right orderable quandle. But, $\mathcal{D}(S^\Artin)$ is a subquandle of $\Conj(\Artin)$, a contradiction.
\end{proof}
 
In the positive direction, it is known that the conjugation quandle $\Conj(G)=\mathcal{D}(G^G)$ of a bi-orderable group $G$ is right-orderable \cite[Proposition 3.4(1)]{MR4450681} and that the free quandle $\mathcal{D}(S^{F(S)})$ is right orderable \cite[Theorem 3.5]{MR4450681}. A free involutory quandle is a free object in the category of involutory quandles. A model for a free involutory quandle is the Dehn quandle $\mathcal{D}(S^{\mathcal{U}})$ of the universal Coxeter group $\mathcal{U}$ on the set $S$. Recall that, $\mathcal{U}$ is simply the free product of $|S|$ many cyclic groups of order two. We have the following result regarding the left orderability of free involuntary quandles.
 
\begin{proposition}\label{free invol left orderable} 
Free involutory quandles are left orderable.
\end{proposition}

\begin{proof}
If $\mathcal{U}$ is the universal Coxeter group generated by $S$, then $\mathcal{D}(S^{\mathcal{U}})$ is the free involutory quandle on $S$. By \cite[Theorem 4.4.4]{Winker1984}, the quandle $\mathcal{D}(S^{\mathcal{U}})$ consists of left-associated products of the form
\begin{equation*}
x_n *x_{n-1}*\cdots *x_1\,,
\end{equation*}
where $n\geq1$ and $x_1,x_2,\ldots,x_n\in S$ with $x_i\neq x_{i+1}$. By \cite[Lemma 4.4.2]{Winker1984}, the multiplication of such products is given by 
\begin{equation}\label{prod-free-inv-quandle}
(x_n*x_{n-1}*\cdots*x_1)*(y_m*y_{m-1}*\cdots*y_1) = x_n*x_{n-1}*\cdots*x_1*y_1*\cdots*y_{m-1}*y_m*y_{m-1}*\cdots* y_1.
\end{equation}
Let $F(S)$ be the free group on $S$ and $\phi:\mathcal{D}(S^{\mathcal{U}})\to\Core(F(S))$ be the map defined by 
\begin{equation*}
\phi(x_n*x_{n-1}*\cdots*x_1) = x_1^{d_1}x_2^{d_2}\cdots x_{n-1}^{d_{n-1}}x_n^{d_n}x_{n-1}^{d_{n-1}}\cdots x_2^{d_2}x_1^{d_1}\,,
\end{equation*} 
where $x_1,x_2,\ldots,x_n\in S$ with $x_i\neq x_{i+1}$ and $d_i=(-1)^{i+1}$ for $1\leq i\leq n$. Using \eqref{prod-free-inv-quandle}, it can be checked that $\phi$ is a quandle homomorphism. Further, it turns out that $\phi$ is injective, and hence we have an embedding of $\mathcal{D}(S^{\mathcal{U}})$ into $\Core(F(S))$. Since free groups are bi-orderable \cite{Vinogradov1949}, it follows from \cite[Proposition 3.4]{MR4450681} that the quandle $\Core(F(S))$ is left orderable. Hence, the free involutory quandle $\mathcal{D}(S^{\mathcal{U}})$ is left orderable.
\end{proof}

\begin{proposition}\label{biordered_alexander_quandle}
Let $<$ be a bi-ordering on a group $G$ and $\phi\in\Aut(G)$. Then the following statements are equivalent:
\begin{enumerate}[(i)]
\item  $<$ is a bi-ordering on $\Alex(G,\phi)$.
\item $1<\phi(g)<g$ for all $g \in G$ with $1 <g$, where $1$ is the identity of $G$.
\end{enumerate}
\end{proposition}

\begin{proof}
For $(i) \Rightarrow (ii)$, let $g\in G$ be such that $1<g$. Since $<$ is a bi-ordering on $\Alex(G,\phi)$, we have $1< g*1$ and $1<1*g$. This gives $1<\phi(g)$ and $1<\phi(g)^{-1}g$, and hence
$1<\phi(g)<g$.
\par
For $(ii) \Rightarrow (i)$, let $x, y, z \in G$ be such that $x<y$. This gives $1<yx^{-1}$, and hence $1<\phi(y)\phi(x)^{-1}<yx^{-1}$. Since $<$ is a bi-ordering on $G$, it follows that $\phi(x) < \phi(y)$ and  $\phi(x)^{-1}x < \phi(y)^{-1} y$. Again using bi-ordering on $G$, we get  $\phi(x)\phi(z)^{-1}z<\phi(y)\phi(z)^{-1}z$ and $\phi(z)\phi(x)^{-1}x < \phi(z)\phi(y)^{-1}y$. By definition, this gives
$x*z<y*z$ and $z*x<z*y$, which is desired.
\end{proof}

An automorphism satisfying statement (ii) of Proposition \ref{biordered_alexander_quandle} is clearly fixed-point free. By Proposition \ref{alexander_quandle_is_dehn_quandle}, generalised Alexander quandles with respect to fixed-point free automorphisms of groups are Dehn quandles. Thus, Proposition \ref{biordered_alexander_quandle} gives examples of bi-orderable Dehn quandles.
\medskip

\section{Dehn quandles of orientable surfaces}\label{section Dehn quandles of surfaces}
Let $S_{g,p}$ be an orientable surface of genus $g$ with $p$ punctures (marked points). The {\it mapping class group} $\mathcal{M}_{g,p}$ of $S_{g,p}$ is defined as the set of isotopy classes of orientation preserving self-homeomorphisms which permute the set $P$ of punctures. We will frequently refer to \cite{Farb-Margalit2012} for related results on mapping class groups. In this section, we investigate Dehn quandle of $S_{g,p}$.
\par

A {\it simple closed curve} in $S_{g,p}$ is an embedding of a circle into the interior of the surface. We say that such a curve is {\it essential} if it is not homotopic to a point or a puncture. A {\it simple closed arc} is an embedding $\sigma:[0,1]\to S_{g,p}$ of a closed interval into the surface such that $\sigma^{-1}(P)=\{0,1\}$. Such an arc is {\it essential} if $\sigma(0)\neq \sigma(1)$. Throughout, by a simple closed curve or an arc, we mean an essential simple closed curve or an arc. For simplicity, we avoid writing $p$ in the notations whenever it is zero. By abuse of notation we use the same symbol to denote the isotopy class and a representative of a curve and an arc.
\par

Given a simple closed curve $\alpha$ and a simple closed arc $\beta$ on $S_{g,p}$, we denote the right hand Dehn twist along $\alpha$ by $T_\alpha$ and the anti-clockwise half twist about the arc $\beta$ by $H_\beta$.

\begin{lemma}\label{fact1}
Let $S_{g,p}$ be a closed orientable surface of genus $g$ with $p$ punctures.
\begin{enumerate}[(i)]
\item Let $\alpha,\alpha'$ be simple closed curves in $S_{g,p}$. Then $T_\alpha$ is isotopic to $T_{\alpha'}$ if and only if $\alpha$ is isotopic to $\alpha'$.
\item Let $\beta,\beta'$ be simple closed arcs in $S_{g,p}$. Then $H_\beta$ is isotopic to $H_{\beta'}$ if and only if $\beta$ is isotopic to $\beta'$.
\end{enumerate}
\end{lemma}

\begin{proof}
The first assertion is proved in \cite[Section 3.3]{Farb-Margalit2012}. By \cite[Section 1.6.2]{Kassel-Turaev2008}, if $\beta$ is isotopic to $\beta'$, then $H_\beta$ is isotopic to $H_{\beta'}$. The converse implication is proved in \cite[Lemma 4.1]{kamadamatsumoto}. However, we provide a shorter alternate proof here for the sake of completeness. Suppose that $H_\beta$ is isotopic to $H_{\beta'}$. Then $(H_{\beta})^2=T_\alpha$ is isotopic to $(H_{\beta'})^2=T_{\alpha'}$, where $\alpha$ and $\alpha'$ are simple closed curves enclosing regular neighbourhoods of $\beta$ and $\beta'$, respectively. It follows from assertion $(i)$ that $\alpha$ is isotopic to $\alpha'$. Choose $\gamma$ in the isotopy class of $\alpha$ (and hence $\alpha'$) such that it does not intersect $\beta$ and $\beta'$. Cutting the surface along $\gamma$, we see that either $\beta$ and $\beta'$ lie in different components of $\overline{S_{g,p} \setminus \gamma}$ or they both lie in the disk component with two punctures (the one bounded by $\gamma$). The former case implies that $H_{\beta}$ is not isotopic to $H_{\beta'}$, a contradiction. Hence, $\beta$ and $\beta'$ lie in the same disk with two punctures, and therefore must be isotopic as closed arcs.
\end{proof}

Let $\dq_{g,p}$ denote the set of isotopy classes of all simple closed curves and simple closed arcs in $S_{g,p}$. A simple closed curve $\alpha$ is said to be \textit{non-separating} if $\overline{S_{g,p}\setminus \alpha}$ is connected, and called \textit{separating} otherwise. Separating and non-separating simple closed arcs are defined similarly. Let $\dq_{g,p}^{ns}$ be its subset consisting of isotopy classes of all {\it non-separating} simple closed curves and simple closed arcs in $S_{g,p}$. In view of Lemma \ref{fact1}, for each isotopy class $y \in \dq_{g,p}$ of a simple closed curve, we can define $T_y$ to be the isotopy class of the right hand Dehn twist along any simple closed curve representing $y$. Similarly, if $y$ represents an isotopy class of a simple closed arc, then we can define $H_y$ to be the isotopy class of the anti-clockwise half twist about the punctures joined by any arc representing $y$. Thus, Lemma \ref{fact1} defines an injective map
$$ \tau: \dq_{g,p} \hookrightarrow \mathcal{M}_{g,p}$$
by setting
$$
\tau(x)=\begin{cases}
T_x \text{ if $x$ is the isotopy class of a simple closed curve,}\\
H_x \text{ if $x$ is the isotopy class of a simple closed arc.}
\end{cases}
$$

We will often identify elements of $\dq_{g,p}$ with corresponding Dehn twists and half twists via $\tau$. The following assertion follows from \cite[Corollary 4.15]{Farb-Margalit2012}.

\begin{lemma}\label{generation-transtitivity}
For each $g,p\geq 0$, the group $\mcg_{g,p}$ is generated by finitely many Dehn twists about non-separating simple closed curves and half twists about simple closed arcs. 
\end{lemma}

\begin{lemma}\label{fact2}
The following hold for each $f \in\mcg_{g,p}$.
\begin{enumerate}[(i)]
\item If $x$ is the isotopy class of a simple closed curve in $S_{g,p}$, then $f T_x f^{-1}=T_{f(x)}$.
\item If $y$ is the isotopy class of a simple closed arc in $S_{g,p}$, then $f H_y f^{-1}=H_{f(y)}$.
\item $\mcg_{g,p}$ acts transitively on the set of isotopy classes of simple closed curves of the same type (separating/non-separating) and the set of isotopy classes of simple closed arcs. 
\end{enumerate}
\end{lemma}

\begin{proof}
The first assertion is proved in \cite[Section 3.3]{Farb-Margalit2012}. By definition, a half twist along an arc acts non-trivially only in a neighbourhood of the arc. We can always choose the neighbourhood such that it does not interest any boundary or a puncture. Then, the same argument as in assertion (i) proves assertion (ii) as well. The third assertion for curves follows from \cite[Section 1.3.1]{Farb-Margalit2012} and for arcs follows from \cite[Section 1.3.3]{Farb-Margalit2012}.
\end{proof}

By Lemma \ref{fact2}, the set of all Dehn twists along non-separating simple closed curves forms one conjugacy class in $\mcg_{g,p}$, whereas the set of all half twists along simple closed arcs forms another conjugacy class in $\mcg_{g,p}$. If $S$ is a generating set for $\mcg_{g,p}$ consisting of generators as in Lemma \ref{generation-transtitivity} together with separating simple closed curves, then $\dq_{g,p}=S^{\mcg_{g,p}}$ as sets. Hence, $\dq_{g,p}$ has the structure of the Dehn quandle of the group $\mcg_{g,p}$ with respect to $S$, and is called the Dehn quandle of the surface $S_{g,p}$. Note that $\dq_{g,p}$ contains $\dq_{g,p}^{ns}$ as a subquandle. To be explicit, by Lemma \ref{fact2}(i)-(ii), the quandle operation in $\dq_{g,p}$ is given as
$$x*y=\begin{cases}
T_y(x) \text{ if $y$ is the isotopy class of a simple closed curve,}\\
H_y(x) \text{ if $y$ is the isotopy class of a simple closed arc.}
\end{cases}$$
Further, in view of Lemma \ref{fact2}(i)-(ii), the map $\tau$ becomes an embedding of quandles as
$$\tau(x * y)=\tau({A_y(x)})= A_{A_y(x)}= A_y A_x A_y^{-1}= A_x * A_y=\tau(x)* \tau(y),$$
where $A=T$ or $H$ depending on whether $x, y$ are simple closed curves or closed arcs.
\par
The construction of the Dehn quandle of a surface $S_{g,p}$ for $p=0$ first appeared in the work of Zablow \cite{Zablow1999, Zablow2003}. Further, \cite{kamadamatsumoto,Yetter2003} considered the quandle structure on the set of isotopy classes of simple closed arcs in $S_{g,p}$ for $p\geq 2$, and called it {\it quandle of cords}. In general, the quandle of cords is a subquandle of $\mathcal{D}_{g,p}$. We can define the Dehn quandle of an orientable surface with boundary components in a similar fashion. In the case of a disk with $n+1$ punctures, this quandle can be identified with the Dehn quandle of the braid group $B_{n+1}$ with respect to its standard set of generators, that is, half twists along the cords. A presentation for the quandle of cords of the plane and the 2-sphere is given in \cite{kamadamatsumoto}. To keep our discussions less technical we will exclude surfaces with boundary components.
\medskip

By \cite[Section 4.4.4]{Farb-Margalit2012}, a generating set for the mapping class group $\mcg_{g, p}$ for $g, p \ge 0$ is the set of appropriate twists along the curves and the arcs of the set 
$$\{a_1,a_2,b_1,b_2,\ldots,b_g,c_1,c_2, \ldots,c_{g-1},a_{12},a_{23},\ldots,a_{(p-1)p},\sigma_{12},\sigma_{23},\ldots,\sigma_{(p-1)p}\},$$ representatives of which are shown in Figure \ref{generators-with-punctures}. A generating set for $\mcg_g$ is given by the set of Dehn twists along the curves of the set $\{a_1,a_2,b_1,b_2,\ldots,b_g,c_1,c_2, \ldots,c_{g-1} \}$ as shown in Figure \ref{Humphries generators}.

\begin{figure}[hbt!]
\begin{center}
\begin{subfigure}{0.993\textwidth}
\centering
\begin{tikzpicture}[scale=0.4]
\begin{knot}[clip width=6, clip radius=4pt]
\strand[-] (0,4.7) to [out=left, in=left, looseness=2.2] (0,-3.6) to [out=right, in=left, looseness=1.3] (3,-3) to [out=right, in=left, looseness=1.3] (6,-3.6) to [out=right, in=left, looseness=1.3] (9,-3) to [out=right, in=left, looseness=1.3] (12,-3.6) to [out=right, in=left, looseness=1.3] (16,-3) to [out=right, in=left, looseness=1.3] (20,-3.6) to [out=right, in=right, looseness=2.2] (20,4.7) to [out=left, in=right, looseness=1.3] (16,4.1) to [out=left, in=right, looseness=1.3] (12,4.7) to [out=left, in=right, looseness=1.3] (9,4.1) to [out=left, in=right, looseness=1.3] (6,4.7) to [out=left, in=right, looseness=1.3] (3,4.1) to [out=left, in=right, looseness=1.3] (0,4.7);
\end{knot};

\begin{scope}
\clip (-1,0.5)rectangle(1,1);
\draw (0,0) circle [x radius=9mm, y radius=9mm];
\end{scope}
\begin{scope}[shift={(0,1.7)}]
\clip (-1.1,-0.8)rectangle(1.1,-1.5);
\draw (0,0) circle [x radius=14mm, y radius=14mm];
\end{scope}
\begin{scope}[shift={(6,0)}]
\clip (-1,0.5)rectangle(1,1);
\draw (0,0) circle [x radius=9mm, y radius=9mm];
\end{scope}
\begin{scope}[shift={(6,1.7)}]
\clip (-1.1,-0.8)rectangle(1.1,-1.5);
\draw (0,0) circle [x radius=14mm, y radius=14mm];
\end{scope}
\begin{scope}[shift={(12,0)}]
\clip (-1,0.5)rectangle(1,1);
\draw (0,0) circle [x radius=9mm, y radius=9mm];
\end{scope}
\begin{scope}[shift={(12,1.7)}]
\clip (-1.1,-0.8)rectangle(1.1,-1.5);
\draw (0,0) circle [x radius=14mm, y radius=14mm];
\end{scope}
\begin{scope}[shift={(20,0)}]
\clip (-1,0.5)rectangle(1,1);
\draw (0,0) circle [x radius=9mm, y radius=9mm];
\end{scope}
\begin{scope}[shift={(20,1.7)}]
\clip (-1.1,-0.8)rectangle(1.1,-1.5);
\draw (0,0) circle [x radius=14mm, y radius=14mm];
\end{scope}

\draw (0,0.6) circle [x radius=20mm, y radius=14mm];
\draw (6,0.6) circle [x radius=20mm, y radius=14mm];
\draw (12,0.6) circle [x radius=20mm, y radius=14mm];
\draw (20,0.6) circle [x radius=20mm, y radius=14mm];

\begin{scope}
\clip (0,0.5)rectangle(-0.7,-3.7);
\draw (0,-1.65) circle [x radius=4mm, y radius=19.5mm];
\end{scope}
\begin{scope}
\clip (0,0.5)rectangle(0.7,-3.7);
\draw[dashed] (0,-1.65) circle [x radius=4mm, y radius=19.5mm];
\end{scope}
\begin{scope}[shift={(6,0)}]
\clip (0,0.5)rectangle(-0.7,-3.7);
\draw (0,-1.65) circle [x radius=4mm, y radius=19.5mm];
\end{scope}
\begin{scope}[shift={(6,0)}]
\clip (0,0.5)rectangle(0.7,-3.7);
\draw[dashed] (0,-1.65) circle [x radius=4mm, y radius=19.5mm];
\end{scope}

\begin{scope}
\clip (0.7,0)rectangle(5.3,0.5);
\draw[dashed] (3,0.5) circle [x radius=23mm, y radius=3mm];
\end{scope}
\begin{scope}
\clip (0.7,0.5)rectangle(5.3,1);
\draw (3,0.5) circle [x radius=23mm, y radius=3mm];
\end{scope}
\begin{scope}[shift={(6,0)}]
\clip (0.7,0)rectangle(5.3,0.5);
\draw[dashed] (3,0.5) circle [x radius=23mm, y radius=3mm];
\end{scope}
\begin{scope}[shift={(6,0)}]
\clip (0.7,0.5)rectangle(5.3,1);
\draw (3,0.5) circle [x radius=23mm, y radius=3mm];
\end{scope}
\begin{scope}
\clip (12,0)rectangle(15,0.5);
\draw[dashed] (16,0.5) circle [x radius=33mm, y radius=3mm];
\end{scope}
\begin{scope}
\clip (12,0.5)rectangle(15,1);
\draw (16,0.5) circle [x radius=33mm, y radius=3mm];
\end{scope}
\begin{scope}
\clip (17,0)rectangle(20,0.5);
\draw[dashed] (16,0.5) circle [x radius=33mm, y radius=3mm];
\end{scope}
\begin{scope}
\clip (17,0.5)rectangle(20,1);
\draw (16,0.5) circle [x radius=33mm, y radius=3mm];
\end{scope}

\begin{scope}
\clip (21.4,3.2)rectangle(23.6,0.04);
\draw (20,0.5) circle [x radius=35mm, y radius=29mm];
\end{scope}
\begin{scope}
\clip (20.5,-2.37)rectangle(22.4,-1.64);
\draw (20,0.5) circle [x radius=35mm, y radius=29mm];
\end{scope}

\begin{scope}[shift={(21.9,2.3)},rotate=55]
\clip (-2.16,0)rectangle(-0.6,0.41);
\draw (0,0) circle [x radius=21.6mm, y radius=4mm];
\end{scope}
\begin{scope}[shift={(21.9,2.3)},rotate=55]
\clip (0.3,0)rectangle(2.16,0.41);
\draw (0,0) circle [x radius=21.6mm, y radius=4mm];
\end{scope}
\begin{scope}[shift={(21.9,2.3)},rotate=55]
\clip (-2.16,0)rectangle(-0.4,-0.41);
\draw[dashed] (0,0) circle [x radius=21.6mm, y radius=4mm];
\end{scope}
\begin{scope}[shift={(21.9,2.3)},rotate=55]
\clip (0.6,0)rectangle(2.16,-0.41);
\draw[dashed] (0,0) circle [x radius=21.6mm, y radius=4mm];
\end{scope}
\begin{scope}[shift={(22.95,1.05)},rotate=15]
\clip (-2.33,0)rectangle(2.33,0.41);
\draw (0,0) circle [x radius=23.3mm, y radius=4mm];
\end{scope}
\begin{scope}[shift={(22.95,1.05)},rotate=15]
\clip (-2.33,0)rectangle(-0.65,-0.41);
\draw[dashed] (0,0) circle [x radius=23.3mm, y radius=4mm];
\end{scope}
\begin{scope}[shift={(22.95,1.05)},rotate=15]
\clip (0.6,0)rectangle(2.33,-0.41);
\draw[dashed] (0,0) circle [x radius=23.3mm, y radius=4mm];
\end{scope}
\begin{scope}[shift={(21.3,-1.45)},rotate=300]
\clip (-2.21,0)rectangle(-0.25,0.41);
\draw (0,0) circle [x radius=21.1mm, y radius=4mm];
\end{scope}
\begin{scope}[shift={(21.3,-1.45)},rotate=300]
\clip (0.69,0)rectangle(2.21,0.41);
\draw (0,0) circle [x radius=21.1mm, y radius=4mm];
\end{scope}
\begin{scope}[shift={(21.3,-1.45)},rotate=300]
\clip (-2.21,0)rectangle(-0.5,-0.41);
\draw[dashed] (0,0) circle [x radius=21.1mm, y radius=4mm];
\end{scope}
\begin{scope}[shift={(21.3,-1.45)},rotate=300]
\clip (0.25,0)rectangle(2.21,-0.41);
\draw[dashed] (0,0) circle [x radius=21.1mm, y radius=4mm];
\end{scope}

\node at (16,0.5) {\Large{$\cdots$}}; \node at (21.4,3.13) {\Huge{$\cdot$}}; \node at (23.03,1.9) {\Huge{$\cdot$}}; \node at (23.46,0.04) {\Huge{$\cdot$}}; \node at (22.4,-1.65) {\Huge{$\cdot$}}; \node at (20.5,-2.4) {\Huge{$\cdot$}}; \node at (22.92,-1.105) {$\cdot$}; \node at (23.13,-0.805) {$\cdot$}; \node at (23.28,-0.505) {$\cdot$};

\node at (-0.93,-1.7) {$a_1$}; \node at (5.02,-1.7) {$a_2$};
\node at (0,2.6) {$b_1$}; \node at (6,2.6) {$b_2$}; \node at (12,2.6) {$b_3$}; \node at (20,2.6) {$b_g$};
\node at (3,1.27) {$c_1$}; \node at (9,1.27) {$c_2$}; \node at (14.7,1.27) {$c_3$}; \node at (17,1.27) {$c_{g-1}$};
\node at (23.6,4.4) {$a_{12}$}; \node at (26,1.7) {$a_{23}$}; \node at (23,-3.98) {$a_{(p-1)p}$};
\node at (21.8,2.35) {{\small$\sigma_{12}$}}; \node at (22.75,0.9) {{\small$\sigma_{23}$}}; \node at (20.9,-1.5) {{\small$\sigma_{(p-1)p}$}};
\end{tikzpicture}
\caption{Generators for punctured surface}
\label{generators-with-punctures}
\end{subfigure}
\newline\vskip5mm

\begin{subfigure}{0.993\textwidth}
\centering
\begin{tikzpicture}[scale=0.4]
\begin{knot}[clip width=6, clip radius=4pt]
\strand[-] (0,4) to [out=left, in=left, looseness=2.3] (0,-2.9) to [out=right, in=left, looseness=1.3] (3,-2.3) to [out=right, in=left, looseness=1.3] (6,-2.9) to [out=right, in=left, looseness=1.3] (9,-2.3) to [out=right, in=left, looseness=1.3] (12,-2.9) to [out=right, in=left, looseness=1.3] (16,-2.3) to [out=right, in=left, looseness=1.3] (20,-2.9) to [out=right, in=right, looseness=2.3] (20,4) to [out=left, in=right, looseness=1.3] (16,3.4) to [out=left, in=right, looseness=1.3] (12,4) to [out=left, in=right, looseness=1.3] (9,3.4) to [out=left, in=right, looseness=1.3] (6,4) to [out=left, in=right, looseness=1.3] (3,3.4) to [out=left, in=right, looseness=1.3] (0,4);
\end{knot};

\begin{scope}
\clip (-1,0.5)rectangle(1,1);
\draw (0,0) circle [x radius=9mm, y radius=9mm];
\end{scope}
\begin{scope}[shift={(0,1.7)}]
\clip (-1.1,-0.8)rectangle(1.1,-1.5);
\draw (0,0) circle [x radius=14mm, y radius=14mm];
\end{scope}
\begin{scope}[shift={(6,0)}]
\clip (-1,0.5)rectangle(1,1);
\draw (0,0) circle [x radius=9mm, y radius=9mm];
\end{scope}
\begin{scope}[shift={(6,1.7)}]
\clip (-1.1,-0.8)rectangle(1.1,-1.5);
\draw (0,0) circle [x radius=14mm, y radius=14mm];
\end{scope}
\begin{scope}[shift={(12,0)}]
\clip (-1,0.5)rectangle(1,1);
\draw (0,0) circle [x radius=9mm, y radius=9mm];
\end{scope}
\begin{scope}[shift={(12,1.7)}]
\clip (-1.1,-0.8)rectangle(1.1,-1.5);
\draw (0,0) circle [x radius=14mm, y radius=14mm];
\end{scope}
\begin{scope}[shift={(20,0)}]
\clip (-1,0.5)rectangle(1,1);
\draw (0,0) circle [x radius=9mm, y radius=9mm];
\end{scope}
\begin{scope}[shift={(20,1.7)}]
\clip (-1.1,-0.8)rectangle(1.1,-1.5);
\draw (0,0) circle [x radius=14mm, y radius=14mm];
\end{scope}

\draw (0,0.6) circle [x radius=20mm, y radius=14mm];
\draw (6,0.6) circle [x radius=20mm, y radius=14mm];
\draw (12,0.6) circle [x radius=20mm, y radius=14mm];
\draw (20,0.6) circle [x radius=20mm, y radius=14mm];

\begin{scope}
\clip (0,0.5)rectangle(-0.6,-3.2);
\draw (0,-1.33) circle [x radius=4mm, y radius=16.1mm];
\end{scope}
\begin{scope}
\clip (0,0.5)rectangle(0.6,-3.2);
\draw[dashed] (0,-1.33) circle [x radius=4mm, y radius=16.1mm];
\end{scope}
\begin{scope}[shift={(6,0)}]
\clip (0,0.5)rectangle(-0.6,-3.2);
\draw (0,-1.33) circle [x radius=4mm, y radius=16.1mm];
\end{scope}
\begin{scope}[shift={(6,0)}]
\clip (0,0.5)rectangle(0.6,-3.2);
\draw[dashed] (0,-1.33) circle [x radius=4mm, y radius=16.1mm];
\end{scope}

\begin{scope}
\clip (0.7,0)rectangle(5.3,0.5);
\draw[dashed] (3,0.5) circle [x radius=23mm, y radius=3mm];
\end{scope}
\begin{scope}
\clip (0.7,0.5)rectangle(5.3,1);
\draw (3,0.5) circle [x radius=23mm, y radius=3mm];
\end{scope}
\begin{scope}[shift={(6,0)}]
\clip (0.7,0)rectangle(5.3,0.5);
\draw[dashed] (3,0.5) circle [x radius=23mm, y radius=3mm];
\end{scope}
\begin{scope}[shift={(6,0)}]
\clip (0.7,0.5)rectangle(5.3,1);
\draw (3,0.5) circle [x radius=23mm, y radius=3mm];
\end{scope}
\begin{scope}
\clip (12,0)rectangle(15,0.5);
\draw[dashed] (16,0.5) circle [x radius=33mm, y radius=3mm];
\end{scope}
\begin{scope}
\clip (12,0.5)rectangle(15,1);
\draw (16,0.5) circle [x radius=33mm, y radius=3mm];
\end{scope}
\begin{scope}
\clip (17,0)rectangle(20,0.5);
\draw[dashed] (16,0.5) circle [x radius=33mm, y radius=3mm];
\end{scope}
\begin{scope}
\clip (17,0.5)rectangle(20,1);
\draw (16,0.5) circle [x radius=33mm, y radius=3mm];
\end{scope}

\node at (16,0.5) {\Large{$\cdots$}}; \node at (-0.93,-1.7) {$a_1$}; \node at (5.02,-1.7) {$a_2$}; 
\node at (0,2.6) {$b_1$}; \node at (6,2.6) {$b_2$}; \node at (12,2.6) {$b_3$}; \node at (20,2.6) {$b_g$};
\node at (3,1.27) {$c_1$}; \node at (9,1.27) {$c_2$}; \node at (14.7,1.27) {$c_3$}; \node at (17,1.27) {$c_{g-1}$};
\end{tikzpicture}
\caption{Humphries generators}
\label{Humphries generators}
\end{subfigure}
\newline\vskip5mm

\begin{subfigure}{0.993\textwidth}
\centering
\begin{tikzpicture}[scale=0.4]
\begin{knot}[clip width=6, clip radius=4pt]
\strand[-] (0,4) to [out=left, in=left, looseness=2] (0,-2.9) to [out=right, in=left, looseness=1.3] (3,-2.3) to [out=right, in=left, looseness=1.3] (6,-2.9) to [out=right, in=left, looseness=1.3] (9,-2.3) to [out=right, in=left, looseness=1.3] (12,-2.9) to [out=right, in=left, looseness=1.3] (17,-2.3) to [out=right, in=left, looseness=1.3] (22,-2.9) to [out=right, in=right, looseness=2] (22,4) to [out=left, in=right, looseness=1.3] (17,3.4) to [out=left, in=right, looseness=1.3] (12,4) to [out=left, in=right, looseness=1.3] (9,3.4) to [out=left, in=right, looseness=1.3] (6,4) to [out=left, in=right, looseness=1.3] (3,3.4) to [out=left, in=right, looseness=1.3] (0,4);
\end{knot};

\begin{scope}
\clip (-1,0.5)rectangle(1,1);
\draw (0,0) circle [x radius=9mm, y radius=9mm];
\end{scope}
\begin{scope}[shift={(0,1.7)}]
\clip (-1.1,-0.8)rectangle(1.1,-1.5);
\draw (0,0) circle [x radius=14mm, y radius=14mm];
\end{scope}
\begin{scope}[shift={(6,0)}]
\clip (-1,0.5)rectangle(1,1);
\draw (0,0) circle [x radius=9mm, y radius=9mm];
\end{scope}
\begin{scope}[shift={(6,1.7)}]
\clip (-1.1,-0.8)rectangle(1.1,-1.5);
\draw (0,0) circle [x radius=14mm, y radius=14mm];
\end{scope}
\begin{scope}[shift={(12,0)}]
\clip (-1,0.5)rectangle(1,1);
\draw (0,0) circle [x radius=9mm, y radius=9mm];
\end{scope}
\begin{scope}[shift={(12,1.7)}]
\clip (-1.1,-0.8)rectangle(1.1,-1.5);
\draw (0,0) circle [x radius=14mm, y radius=14mm];
\end{scope}
\begin{scope}[shift={(22,0)}]
\clip (-1,0.5)rectangle(1,1);
\draw (0,0) circle [x radius=9mm, y radius=9mm];
\end{scope}
\begin{scope}[shift={(22,1.7)}]
\clip (-1.1,-0.8)rectangle(1.1,-1.5);
\draw (0,0) circle [x radius=14mm, y radius=14mm];
\end{scope}

\begin{scope}
\clip (2,-2.4)rectangle(3,3.5);
\draw (3,0.54) circle [x radius=4mm, y radius=28.7mm];
\end{scope}
\begin{scope}
\clip (3,3.5)rectangle(4,-2.4);
\draw[dashed] (3,0.54) circle [x radius=4mm, y radius=28.7mm];
\end{scope}
\begin{scope}[shift={(6,0)}]
\clip (2,-2.4)rectangle(3,3.5);
\draw (3,0.54) circle [x radius=4mm, y radius=28.7mm];
\end{scope}
\begin{scope}[shift={(6,0)}]
\clip (3,3.5)rectangle(4,-2.4);
\draw[dashed] (3,0.54) circle [x radius=4mm, y radius=28.7mm];
\end{scope}

\begin{scope}
\clip (14,-2.6)rectangle(15,3.7);
\draw (15,0.54) circle [x radius=4mm, y radius=30.5mm];
\end{scope}
\begin{scope}
\clip (15,3.7)rectangle(16,-2.6);
\draw[dashed] (15,0.54) circle [x radius=4mm, y radius=30.5mm];
\end{scope}
\begin{scope}[shift={(4,0)}]
\clip (14,-2.6)rectangle(15,3.7);
\draw (15,0.54) circle [x radius=4mm, y radius=30.5mm];
\end{scope}
\begin{scope}[shift={(4,0)}]
\clip (15,3.7)rectangle(16,-2.6);
\draw[dashed] (15,0.54) circle [x radius=4mm, y radius=30.5mm];
\end{scope}

\node at (17,0.5) {\Large{$\cdots$}}; \node at (2.1,1.6) {$s_1$}; \node at (8.1,1.6) {$s_2$}; \node at (14.1,1.6) {$s_3$}; \node at (17.76,1.6) {$s_{g-1}$};
\end{tikzpicture}
\caption{Separating curves}
\label{generators_dehn_quandle}
\end{subfigure}
\caption{Generators for mapping class group and Dehn quandle of a surface}
\label{generators_mgc_dq}
\end{center}
\end{figure}
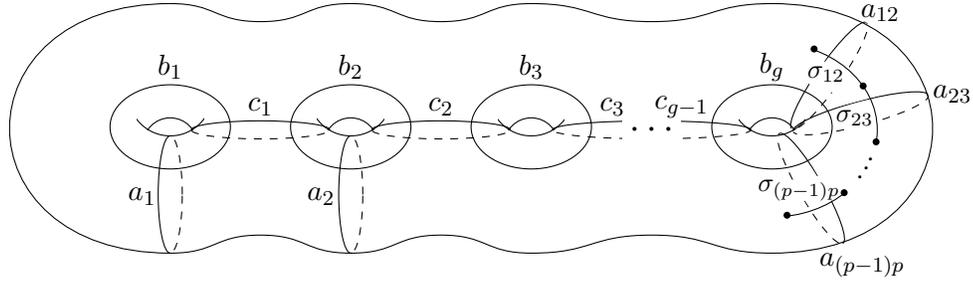
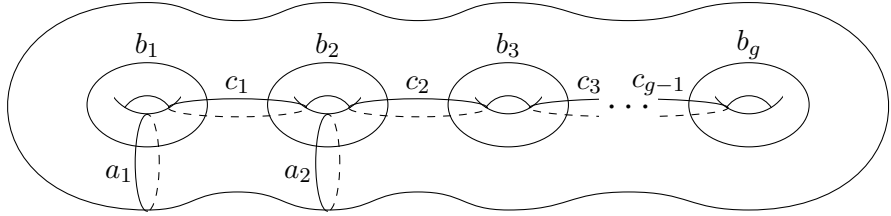
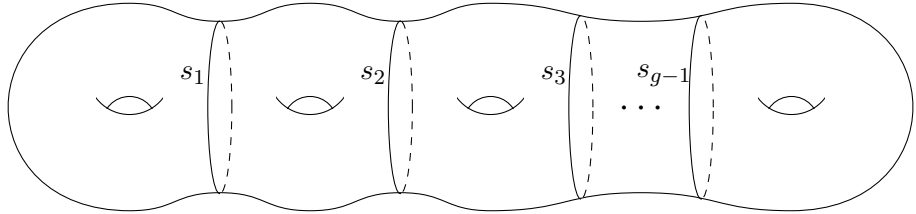

Proposition \ref{generators-dehn-quandle-general} yields the following.

\begin{proposition}\label{cor:gen_Dehnquandle-with-punctures}
For $g, p \ge 0$, the following hold:
\begin{enumerate}[(i)]
\item $\mathcal{D}_{g, p}^{ns}$ is generated by $$\{a_1,a_2,b_1,b_2,\ldots,b_g,c_1,c_2, \ldots,c_{g-1},a_{12},a_{23},\ldots,a_{(p-1)p},\sigma_{12},\sigma_{23},\ldots,\sigma_{(p-1)p} \}.$$
\item $\mathcal{D}_g^{ns}$ is generated by $$\{ a_1,a_2,,b_1,b_2,\ldots,b_g,c_1,c_2,\ldots,c_{g-1} \}.$$
\item $\mathcal{D}_g$ is generated by $$\{a_1,a_2,b_1,b_2,\ldots,b_g,c_1,c_2,\ldots,c_{g-1},s_1,s_2,\ldots,s_{\lfloor \frac{g}{2}\rfloor} \}.$$
\end{enumerate}
Here, the generators for $\mathcal{D}_{g, p}^{ns}$, $\mathcal{D}_g^{ns}$ and $\mathcal{D}_g$ are shown in Figure \ref{generators_mgc_dq}.
\end{proposition}

Next, we give an alternate and short proof of the main result of \cite[Theorem 3.1]{NiebrzydowskiPrzytycki2009}.

\begin{theorem}\label{Niebrzydowski Przytycki theorem}
The knot quandle of the trefoil knot is isomorphic to the Dehn quandle of the torus.
\end{theorem}

\begin{proof}
Recall that the trefoil knot $K$ is prime and its knot group is isomorphic to the braid group $B_3= \langle \sigma_1, \sigma_2 \mid \sigma_1\sigma_2\sigma_1=\sigma_2\sigma_1\sigma_2, \rangle$. Thus, by Corollary \ref{knot-prime-dehn-quandle}, we have $Q(K)\cong \mathcal{D}(\sigma_1^{B_3})$. By \cite[Section 3.6.4]{Farb-Margalit2012}, we have $\mathcal{M}_1= \langle a, b \mid aba=bab, \quad (ab)^6=1 \rangle$. Thus, there is a surjective group homomorphism $\psi: B_3 \to \mathcal{M}_1$ given by $\psi(\sigma_1)=a$ and $\psi(\sigma_2)=b$. Further, $\ker (\psi)= \langle (\sigma_1\sigma_2)^6 \rangle$ is contained in the center of $B_3$. Notice that the restriction $\psi|: \mathcal{D}(\sigma_1^{B_3}) \to \mathcal{D}(a^{\mathcal{M}_1})$ of $\psi$ is a surjective quandle homomorphism. 

\par
We claim that $\psi|$ is injective, and hence an isomorphism. Suppose that $\psi|(\sigma_1^x)=\psi|(\sigma_1^y)$ for $x, y \in B_3$. This implies that $\sigma_1^x ({\sigma_1}^{-1})^y \in \ker (\psi)$, and hence 
\begin{equation}\label{trefoil-dehn}
\sigma_1^x ({\sigma_1}^{-1})^y= (\sigma_1\sigma_2)^{6k}
\end{equation}
 for some $k \in \mathbb{Z}$. But, notice that $\sigma_1^x ({\sigma_1}^{-1})^y=[x, \sigma_1][\sigma_1, y]$ lies in the commutator subgroup of $B_3$, which consists of words in $\sigma_1, \sigma_2$ whose total sum of exponents is zero. In view of \eqref{trefoil-dehn}, this is possible if and only if $k=0$, that is, $\sigma_1^x=\sigma_1^y$. Thus, $$Q(K) \cong \mathcal{D}(\sigma_1^{B_3}) \cong \mathcal{D}(a^{\mathcal{M}_1})= \mathcal{D}_1,$$ as desired.
\end{proof}

If $a$ and $b$ are isotopy classes of simple closed curves in an orientable surface, then their {\it geometric intersection number} $i(a,b)$ is defined to be the minimal number of intersection points between a representative curve in the class $a$ and a representative curve in the class $b$. Low order intersection numbers of simple closed curves can be interpreted in terms of relations in Dehn quandles as follows. 

\begin{lemma}\label{rel_dt_curve}
 If $a$ and $b$ are distinct isotopy classes of simple closed curves, then the following hold:
\begin{enumerate}[(i)]
\item $i(a,b)=0$ if and only if $a*b=a$ and $b*a=b$,
\item $i(a,b)=1$ if and only if $a*b*a=b$ and $b*a*b=a$.
\end{enumerate}
\end{lemma}

\begin{proof}
Recall that $i(a,b)=0$ if and only if $[T_a,T_b]=1$. But the latter relation in terms of the quandle operation is the same as $a*b=a$ and $b*a=b$. Similarly, $i(a,b)=1$ if and only if $T_aT_bT_a=T_bT_aT_b$. In this case, the latter relation is the same as $a*b*a=b$ and $b*a*b=a$.
\end{proof}

The geometric intersection number completely determines two generated subquandles of Dehn quandles of surfaces.

\begin{proposition}\label{structure_2_gen_quandle}
Let $a, b$ be two distinct elements in $\mathcal{D}_g^{ns}$. Then the following hold:
\begin{enumerate}[(i)]
\item If $i(a,b)=0$, then $\langle a, b \rangle$ is the trivial quandle. 
\item If $i(a,b)=1$, then $\langle a, b \rangle$ is the knot quandle of the trefoil knot.
\item If $i(a,b)\geq2$, then $\langle a, b \rangle$ is the free quandle on two generators.
\end{enumerate}
\end{proposition}

\begin{proof}
By Lemma \ref{rel_dt_curve}(i), $i(a,b)=0$ if and only if $a*b=a$ and $b*a=b$. Thus, the subquandle generated by $a, b$ is the trivial quandle on two elements. Again, by Lemma \ref{rel_dt_curve}(ii), $i(a,b)=1$ if and only if $a*b*a=b$ and $b*a*b=a$. In view of Theorem \ref{Niebrzydowski Przytycki theorem}, the subquandle generated by $a, b$ is the knot quandle of the trefoil knot. This proves the first two assertions.
\par
For the third assertion, suppose that $i(a,b)\geq2$. Recall that $\mathcal{D}_g^{ns}$ is a subquandle of $\Conj (\mathcal{M}_g)$. Thus, if there is any non-trivial relation in $\mathcal{D}_g^{ns}$ involving $a$ and $b$, then rewriting this relation in terms of conjugation gives a non-trivial relation in $\mathcal{M}_g$ involving $T_a$ and $T_b$. But, by \cite[Theorem 3.14]{Farb-Margalit2012}, the group generated by $T_a$ and $T_b$ is isomorphic to the free group of rank 2, a contradiction. Hence, $a$ and $b$ generate the free quandle on two generators.
\end{proof}

\begin{proposition}
There is an embedding of quandles $\mathcal{D}_1 \hookrightarrow \mathcal{D}_{g,p}$ for $g \ge 1$ and $p \ge 0$.
\end{proposition}
\begin{proof}
Theorem \ref{Niebrzydowski Przytycki theorem} gives $$\mathcal{D}_1=\langle a,b \mid a*b*a=b, \quad b*a*b=a\rangle.$$ For $g\geq1$ and $p \ge 0$, choose two non-separating simple closed curves $x, y$ in $S_{g,p}$ such that $i(x,y)=1$. Consider the map $\psi: \mathcal{D}_1\to\mathcal{D}_{g,p}$ given by $\psi(a)= x$ and $\psi(b)= y$. Then, by Proposition \ref{structure_2_gen_quandle}(ii), $\psi$ is an isomorphism onto its image.
\end{proof}

Theorem \ref{Niebrzydowski Przytycki theorem} (\cite[Theorem 3.1]{NiebrzydowskiPrzytycki2009}) gives a presentation of the Dehn quandle of the torus. In general, in the follow-up \cite{DhanwaniRaundalSingh2022} of this work, we have given two approaches to write explicit presentations for the class of Dehn quandles.
\par

Regarding orderability of Dehn quandles of surfaces, as a consequence of Proposition \ref{dehn-quandle-not-ordererable}, we have

\begin{corollary}\label{dehn-quandle-not-order}
If $g \ge 0$ and $p \ge 3$ or $g \ge 1$ and $p$ is arbitrary, then the Dehn quandle $\mathcal{D}_{g,p}^{ns}$ is neither right nor left orderable. 
\end{corollary}

\begin{remark}
Most results of this as well as the next section will hold for Dehn quandles of surfaces with boundary components.
\end{remark}

\begin{remark}
We note that, using the work of Gervais \cite{Gervais1996} on central extensions of mapping class groups, Nosaka \cite[Theorem 4]{MR4028088} has determined enveloping groups of Dehn quandles of surfaces of genus more than two and without punctures. At this point we have not been able to determine enveloping groups in case of surfaces with punctures.
\end{remark}

\medskip

\section{Automorphisms of Dehn quandles of orientable surfaces}\label{section auto dehn quandles surfaces}
In this section, we compute automorphism groups of Dehn quandles of surfaces. We say that two simple closed arcs are {\it disjoint} if they don't have common end points and also don't intersect anywhere on the surface.

\begin{lemma}\label{rel_dt_arcs}
 If $a$ and $b$ are non-isotopic simple closed arcs, then the following holds:
\begin{enumerate}[(i)]
\item $a$ and $b$ are disjoint if and only if $a*b=a$ and $b*a=b$.
\item $a$ and $b$ have only one end point in common and do not intersect anywhere on the surface if and only if $a*b*a=b$ and $b*a*b=a$.
\end{enumerate}
\end{lemma}

\begin{proof}
By definition of a half twist about a simple closed arc, there exists a twice punctured disk outside which the half twist act trivially. Let $D_a$ and $D_b$ be such disks for the half twists $H_a$ and $H_b$, respectively. Let $\alpha$ and $\beta$ be boundaries of $D_a$ and $D_b$, respectively. Then $\alpha$ and $\beta$ are simple closed curves and it is known that $H_a^2=T_\alpha$ and $H_b^2=T_\beta.$ 
\par
We now prove assertion (i). If $a$ and $b$ are disjoint, then we can choose $D_a$ and $D_b$ such that they do not intersect. Thus, by definition, $H_a$ and $H_b$ commute. In view of Lemma~\ref{fact2}(ii), we have $H_{H_a(b)}=H_aH_bH_a^{-1}=H_b$. Now, by Lemma \ref{fact1}(ii), $H_a(b)=b$, which in quandle operation can be written as $b*a=b$. Similarly, one can show that $a*b=a$.
\par
For the converse, assume that $a*b=a$. This implies that $H_b(a)=a$, and hence $H_{H_b(a)}=H_a$, which further gives $H_aH_b=H_bH_a$. It follows that $T_\alpha=H_a^2$ commutes with $T_\beta=H_b^2$. Now, Lemma~\ref{rel_dt_curve}(i) gives $i(\alpha,\beta)=0$. This implies that either $\alpha$ and $\beta$ are isotopic or disjoint. If $\alpha$ and $\beta$ are disjoint and not isotopic, then the disks $D_a$ and $D_b$ are disjoint. This implies that the arcs $a$ and $b$ are disjoint, which is desired. Now, suppose that $\alpha$ and $\beta$ are isotopic. We choose $\gamma$ in the isotopy class of $\alpha$ and $\beta$ such that $\gamma$ does not intersect $a$ and $b$. If $a$ and $b$ lie in different connected component of $\overline{S_{g,p}\setminus{\gamma}}$, then $a$ and $b$ are disjoint, and we are done. Now, suppose that $a$ and $b$ lie in the same connected component of $\overline{S_{g,p}\setminus{\gamma}}$. Consider the twice punctured disk $D_\gamma$ which contains both $a$ and $b$. Since upto isotopy there exists only one simple closed arc in a disk with two punctures, it follows that $a$ and $b$ are isotopic, which is a contradiction. This completes the proof of assertion (i).
\par
 
Next we consider assertion (ii). Let $a$ and $b$ are simple closed arcs which share exactly one end point and do not intersect anywhere on the surface. Then it is known that $H_aH_bH_a=H_bH_aH_b$, which can be rewritten as $b*a*b=b$ and $a*b*a=b$. For converse, we prove that half twists $H_a$ and $H_b$ do not satisfy the braid relation for any other intersection pattern of arcs $a$ and $b$. We argue case by case for each intersection pattern.

\begin{enumerate}[(1)]
\item  $a$ and $b$ do not share any end point: Since $H_b(a)$ and $a$ have the same set of end points, it follows that $H_aH_b(a)$ and $a$ also have the same set of end points. Now, if $H_a$ and $H_b$ satisfy the braid relation, then $H_aH_b(a) =b$, which implies that $a$ and $b$ have the same set of end points, a contradiction.
\item $a$ and $b$ share (one or both) end points and $i(a,b)\geq1$ on the surface: If $H_a$ and $H_b$ satisfy the braid relation, then $H_aH_b(a)=b$. This implies that $i(a,b)=i(a, H_aH_b(a))=i(a, H_b(a))$. But, Figure~\ref{braid_rel_arcs_intersecting} shows that $i(a, H_b(a))> i(a,b)$, a contradiction.

\begin{figure}[hbt!]
\centering
\labellist
\small
\pinlabel $a$ at 85 120
\pinlabel $b$ at 68 77
\pinlabel $\beta$ at 15 130
\pinlabel $a$ at 410 150
\pinlabel $b$ at 420 81
\pinlabel $\beta$ at 485 130
\pinlabel $H_b(a)$ at 452 45
\pinlabel $H_b$ at 240 110
\endlabellist
\includegraphics[height=4cm]{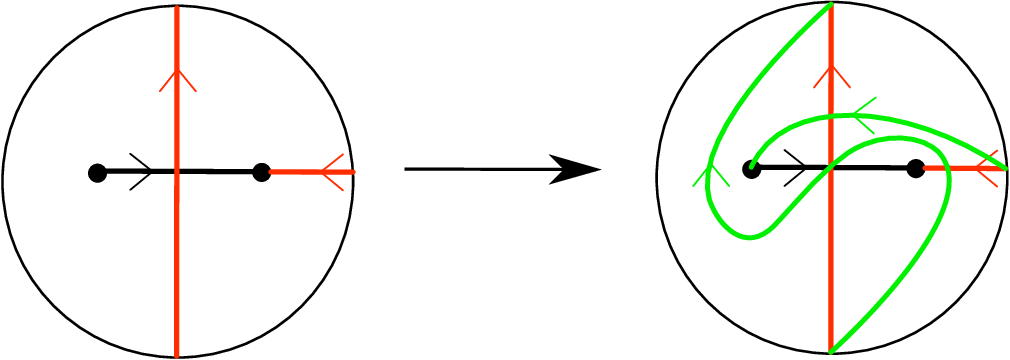}
\caption{Relations amongst half twists}\label{braid_rel_arcs_intersecting}
 \end{figure}
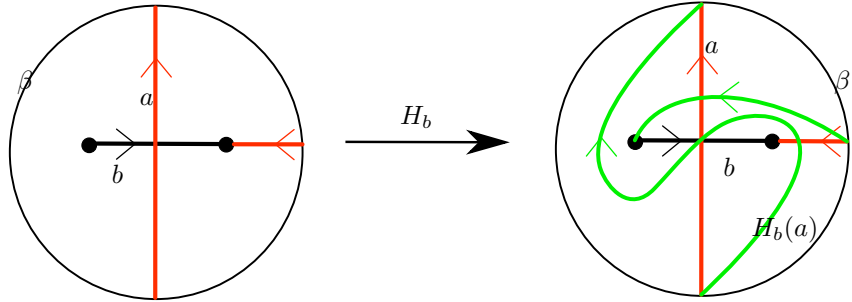

\item  $a$ and $b$ share both end points and $i(a,b)=0$ on the surface: If $H_a$ and $H_b$ satisfy the braid relation, then $i(a, b)=i(a, H_aH_b(a))=i(a, H_b(a))$. But Figure~\ref{braid_rel_arcs_same_end} shows that $i(a, H_b(a))=1$, a contradiction.
\end{enumerate}
 This completes the proof of assertion (ii).
\begin{figure}[hbt!]
\centering
\labellist
\small
\pinlabel $a$ at 155 82
\pinlabel $b$ at 98 82
\pinlabel $\beta$ at 15 130
\pinlabel $a$ at 460 85
\pinlabel $b$ at 395 85
\pinlabel $\beta$ at 485 130
\pinlabel $H_b(a)$ at 420 60
\pinlabel $H_b$ at 240 110
\endlabellist
\includegraphics[height=4cm]{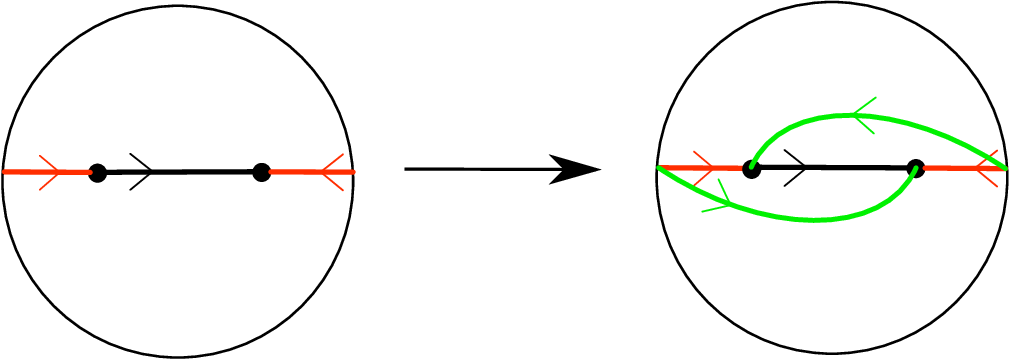}
\caption{Relations amongst half twists}\label{braid_rel_arcs_same_end}
 \end{figure}
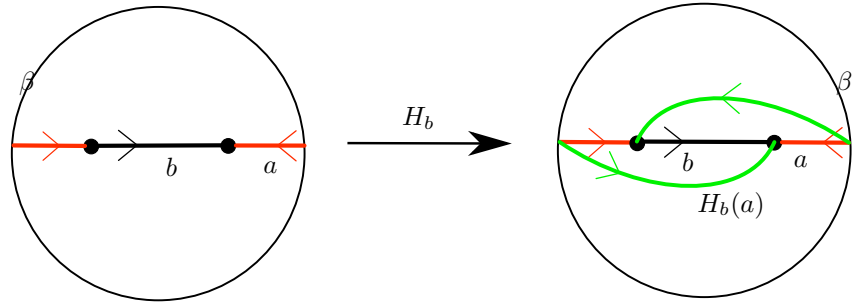
\end{proof}
 
We now determine automorphism groups of Dehn quandles of surfaces.

\begin{lemma}\label{disjoint-curve-arc-commute}
Let $a$ be a simple closed curve and $b$ a simple closed arc in $S_{g,p}.$ Then $a$ and $b$ are disjoint in $S_{g,p}$ if and only if $T_aH_b=H_bT_a$. 
\end{lemma}

\begin{proof}
Since both Dehn twists and half twists are defined locally, if $a$ and $b$ are disjoint, then $T_aH_b=H_bT_a$. For the converse, suppose that  $T_aH_b=H_bT_a$. Let $\beta$ be a simple closed curve bounding a regular neighbourhood of $b$. It follows from the relation $T_aH_b=H_bT_a$ that $T_aT_{\beta}=T_{\beta}T_a$, where $T_{\beta}=H_b^2$. Consequently, we have $i(a,\beta)=0$. Now, if $a$ and $b$ lie in different components of $\overline{S_{g,p}\setminus \beta}$, then they are clearly disjoint.  And, if $a$ and $b$ lie in the same component of $\overline{S_{g,p}\setminus \beta}$ (which is a disk with two punctures), then $a$ being an essential simple closed curve must be isotopic to $\beta$. Thus, $a$ and $b$ are disjoint in this case as well.
\end{proof}

\begin{theorem}\label{Aut-Dehn-quandle-puncture}
The following hold:
\begin{enumerate}[(i)]
\item $\Inn(\dq_{g,p}^{ns})\cong \mcg_{g,p}/\Z(\mcg_{g,p})$ for all $g, p \ge 0$.
\item $\Inn(\dq_{g,p})\cong \mcg_{g,p}/\Z(\mcg_{g,p})$ for all $g, p \ge 0$.
\item  $\Aut(\Conj(\mathcal{M}_{g, p})) \cong \Aut(\mathcal{M}_{g, p})$ if $(g,p) \notin \{ (0,2), (1,0), (1,1), (1,2), (2,0) \}$.
\item  $\Aut(\dq_{g,p}^{ns}) \cong \Inn(\dq_{g,p}^{ns})$ for all $g, p \ge 0$.
\item  $\Aut(\dq_{g,p}) \cong \Inn(\dq_{g,p})$ for all $g, p \ge 0$. 
\end{enumerate}
\end{theorem}

\begin{proof}
Since the underlying sets of quandles $\dq_{g,p}^{ns}$ and $\dq_{g,p}$ generate the group $\mcg_{g,p}$, the first two isomorphisms follows from Proposition \ref{action-corollary}. We know from \cite[Section 3.4]{Farb-Margalit2012} that $\Z(\mathcal{M}_{g, p})=1$ if and only if $(g,p) \notin \{ (0,2), (1,0), (1,1), (1,2), (2,0) \}$. The third isomorphism follows from \cite[Corollary 1]{BardakovNasybullovSingh2019}.
\par
	
We now proceed to prove the last two isomorphisms. We first claim that an automorphism $\phi$ of $\dq_{g,p}^{ns}$ maps closed curves to closed curves and closed arcs to closed arcs. If $g=0$ and $p \ge 0$, then the claim holds trivially since there are no non-separating closed curves. Similarly, the claim holds trivially for $g \ge 1$ and $p=0, 1$ since there are no closed arcs. For $g \ge 1$ and $p\geq2$, a generating set for the quandle $\dq_{g,p}^{ns}$ is $$S=\{a_1,a_2,b_1,b_2,\ldots,b_g,c_1,c_2, \ldots,c_{g-1},a_{12},a_{23},\ldots,a_{(p-1)p},\sigma_{12},\sigma_{23},\ldots,\sigma_{(p-1)p}\}$$ as shown in Figure \ref{generators-with-punctures} (see \cite[Section 4.4.4]{Farb-Margalit2012}). It is enough to prove our claim for elements of this generating set. We note that no half twist form a braid relation with a Dehn twist. Suppose that there is a closed curve in $S$ such that its image under $\phi$ is a closed arc. Note that Dehn twists along any two closed curves in $S$ with geometric intersection number one satisfy the braid relation, and the subset of closed curves in $S$ forms a chain. Since $\phi$ is a quandle homomorphism, it preserves braid relations, and hence maps all closed curves in $S$ onto closed arcs. If $p=2$, then half twists about closed arcs $\phi(c_{g-1})$ and $\phi(b_g)$ satisfy the braid relation, where $c_{g-1}=a_1$ when $g=1$. Since there are only two punctures, the two arcs $\phi(c_{g-1})$ and $\phi(b_g)$ must have both the end points in common, a contradiction. Now, we assume that $p \ge 3$. Consider the set $\{\phi(c_{g-1}),\phi(b_g),\phi(a_{12}),\phi(a_{23}) \}$ of closed arcs. Then half twists about arcs $\phi(c_{g-1}),\phi(a_{12})$ and $\phi(a_{23})$ mutually commute with each other, which implies that these three arcs are mutually disjoint. At the same time, half twist about each of these three arcs satisfy the braid relation with the half twist about $\phi(b_g)$, which implies that each of the three arcs has a common end point with $\phi(b_g)$, a contradiction. Hence, $\phi$ must map closed curves to closed curves. A similar argument shows that $\phi$ maps closed arcs to closed arcs.
	
\par
If $\phi$ is any quandle automorphism of $\dq_{g,p}^{ns}$, then it preserves quandle relations as in lemmas \ref{rel_dt_curve} and \ref{rel_dt_arcs}. Thus, elements of $\phi(S)$ will have the same intersection pattern as of the starting generating set $S$. In the case $g=0$, cutting the punctured 2-sphere along the  generating simple closed arcs $\{\sigma_{12},\ldots,\sigma_{(p-1)p}\}$ and their images $\{\phi(\sigma_{12}),\ldots,\phi(\sigma_{(p-1)p})\}$, we get two homeomorphic surfaces. Gluing back the boundaries induces an orientation preserving homeomorphism $f$ of the punctured 2-sphere that maps the ordered tuple $\{\phi(\sigma_{12}),\ldots,\phi(\sigma_{(p-1)p})\}$ onto $\{\sigma_{12},\ldots,\sigma_{(p-1)p}\}$. For $g \ge 1$, we appeal to the change of co-ordinates principle \cite[Section 1.3.2]{Farb-Margalit2012}, which states that any two collections of simple closed curves with the same intersection pattern can be mapped onto each other by an orientation-preserving homeomorphism of the surface. See also \cite[Sections 4 and 8]{Ivanov}. Thus, there exists a $f\in\mathcal{M}_{g,p}$ such that the ordered tuples 
\begin{small}
$$\big(f\phi(a_1),f\phi(a_2),f\phi(b_1),\ldots,f\phi(b_g),f\phi(c_1),\ldots,f\phi(c_{g-1}), f\phi(a_{12}),\ldots,f\phi(a_{(p-1)p}) \big)$$ 
\end{small}
and $$(a_1,a_2,b_1,\ldots,b_g,c_1,\ldots,c_{g-1},a_{12},\ldots,a_{(p-1)p})$$ are the same. It remains to prove that $f\phi(\sigma_{i(i+1)})=\sigma_{i(i+1)}$ for each $1\leq i\leq p-1$. Since $\phi$ and $f$ both preserve intersection patterns of curves and arcs, so does $f \phi$. Observe that the arc $\sigma_{12}$ is disjoint from each curve in $S$ except the curve $a_{12}$. Thus, $f\phi(\sigma_{12})$ is disjoint from each curve in $S$ except the curve $f\phi(a_{12})=a_{12}$. Cutting the  surface along the curves $\{a_1,b_1,b_2,\ldots,b_g,c_1,c_2, \ldots,c_{g-1}, a_{23}\}$, we see that $\sigma_{12}$ and $f\phi(\sigma_{12})$ both lie in the same subsurface of $S_{g,p}$, which is a disk with two punctures. Hence, $\sigma_{12}$ and $f\phi(\sigma_{12})$  must be isotopic. By a similar argument, for each $2\leq i\leq p-2$, cutting the surface along the curves $\{a_{(i-1)i},b_g,a_{(i+1)(i+2)}\}$, we see that both $\sigma_{i(i+1)}$ and $f\phi(\sigma_{i(i+1)})$ lie in a disk with two punctures, and hence are isotopic. Finally, cutting the surface along the curves  $\{a_2,b_2,\ldots,b_g,c_2, \ldots,c_{g-1},a_{(p-2)(p-1)}\}$, we see that $\sigma_{(p-1)p}$ and $f\phi(\sigma_{(p-1)p})$ lie in a disk with two punctures, and therefore are isomorphic. Thus, it follows that the ordered tuples 
\begin{small}
$$\big(f\phi(a_1),f\phi(a_2),f\phi(b_1),\ldots,f\phi(b_g),f\phi(c_1),\ldots,f\phi(c_{g-1}), f\phi(a_{12}),\ldots,f\phi(a_{(p-1)p}),f\phi(\sigma_{12}),\ldots,f\phi(\sigma_{(p-1)p}) \big)$$ 
\end{small}
and 
$$(a_1,a_2,b_1,\ldots,b_g,c_1,\ldots,c_{g-1},a_{12},\ldots,a_{(p-1)p},\sigma_{12},\ldots,\sigma_{(p-1)p})$$ are the same. Thus, $f \phi$ is identity on the quandle $\dq_{g,p}^{ns}$, and hence $\phi=f^{-1}$. Since the center acts trivially on the system of simple closed curves and arcs, we can take $f \in \mathcal{M}_{g,p}/ \Z(\mathcal{M}_{g,p})$. But $\mathcal{M}_{g,p}/ \Z(\mathcal{M}_{g,p}) \cong \Inn(\mathcal{D}_{g,p}^{ns})$ by assertion (i), and hence $\phi$ is an inner automorphism. The last isomorphism follows along similar lines.
\end{proof}

As a consequence of Theorem \ref{Aut-Dehn-quandle-puncture}, we can determine connectivity of Dehn quandles of surfaces.

\begin{proposition}\label{o(dq)}
For $g \ge 1$ and $p \ge 2$, $\mathcal{D}_g^{ns}$ is connected, whereas $\mathcal{D}_g$ ($g \ge 2$), $\mathcal{D}_{g,p}$ and $\mathcal{D}_{g,p}^{ns}$ are not connected.
\end{proposition}

\begin{proof}
It is known that given two non-separating simple closed curves $a$ and $b$ on an orientable surface, there is an element $f \in \Inn(\mathcal{D}_g^{ns}) \cong \mcg_g/\Z(\mcg_g)$ such that $f(a)=b$. Hence, $\dq_g^{ns}$ is connected. On the other hand, a non-separating curve cannot be mapped to a separating curve, and hence $\mathcal{D}_g$ is not connected. Similarly, a simple closed curve cannot be mapped to a simple closed arc by a mapping class, and hence $\mathcal{D}_{g,p}$ and $\mathcal{D}_{g,p}^{ns}$ are not connected for $p\geq 2$.
 \end{proof}
 
The next result shows that Dehn quandles determine surfaces with a few common exceptions.

\begin{proposition}\label{dehn-surface-rigid}
Let $S_g$ and $S_{g'}$ be surfaces such that $g, g' \ge 3$. Then $S_g$ is homeomorphic to $S_{g'}$  if and only if   $\mathcal{D}_g^{ns} \cong \mathcal{D}_{g'}^{ns}$.
\end{proposition}

\begin{proof}
The forward implication is obvious. For the converse, since $\mathcal{D}_g^{ns} \cong \mathcal{D}_{g'}^{ns}$, it follows from \cite[Section 3.4]{Farb-Margalit2012} and Theorem \ref{Aut-Dehn-quandle-puncture}(i) that $\mcg_g \cong \Inn(\mathcal{D}_g^{ns}) \cong \Inn(\mathcal{D}_{g'}^{ns}) \cong \mcg_{g'}$. Now, \cite[Theorem A.1]{RafiSchleimer2011} implies that $S_g$ is homeomorphic to $S_{g'}$.
\end{proof}

We note that Proposition \ref{dehn-surface-rigid} also holds for $\mathcal{D}_g$ in place of $\mathcal{D}_g^{ns}$.
\medskip

\section{Epimorphisms from Dehn quandles onto homological quandles of surfaces}\label{sec-homological-quandles}
This section is motivated by the work \cite{Yetter2003} of Yetter. Recall that an element $x$ of a group $G$ is said to be {\it primitive} if there do not exist $y\in G$ and integer $k$ with $|k| \ge 2$ such that $y^k=x$. Further, if the order of $x$ is finite, then the order of $y$ must be $k$ times the order of $x$.
\par

Let $S_g$ be a closed orientable surface of genus $g \ge 1$. For isotopy classes $a$ and $b$ of transverse, oriented, simple closed curves in $S_g$, the {\it algebraic intersection number} $\hat{i}(a,b)$ is defined as the sum of the indices of the intersection points of $a$ and $b$, where an intersection point is of index +1 when the orientation of the intersection agrees with the orientation of the surface, and is -1 otherwise. In other words, index at an intersection point is +1 if  standing on an arc of $a$, an arc of $b$ passes from left to right.

It turns out that $\hat{i}(a,b)$ depends only on homology classes (and hence isotopy classes) of $a$ and $b$. Recall that $\mathrm{H}_1(S_g, \mathbb{Z}) \cong \mathbb{Z}^{2g}$ and $\mathrm{H}_1(S_g,\mathbb{Z}_n)\cong \mathbb{Z}_n^{2g}$ for each $n\geq 2$. We view elements of these groups as row vectors. It is known that $\hat{i}(-, -)$ gives a skew-symmetric (in fact, symplectic) bilinear form on the $\mathbb{Z}$-module $\mathrm{H}_1(S_g,\mathbb{Z})$ (see \cite[Section 6.1.2]{Farb-Margalit2012}). Then, as in \cite[Example 4]{Yetter2003}, the binary operation given by
\begin{equation}
x*y=x + \hat{i}(x,y) y
\end{equation}
for $x, y \in \mathrm{H}_1(S_g, \mathbb{Z})$, gives a quandle structure on $\mathrm{H}_1(S_g, \mathbb{Z})$. 
\par

Let $\mathcal{P}'_g$ denote the set of all primitive elements in $\mathrm{H}_1(S_g, \mathbb{Z})$ and $\mathcal{P}'_{g,n}$ denote the set of all primitive elements in $\mathrm{H}_1(S_g,\mathbb{Z}_n)$ for each $n\geq 2$. If $x, y \in \mathcal{P}'_g$, then $x*y \in \mathcal{P}'_g$. For, if $x*y$ is not primitive, then $x=(x*y)*^{-1}y$ is not primitive, a contradiction. Thus, $(\mathcal{P}'_g, *)$ is a subquandle of $( \mathrm{H}_1(S_g, \mathbb{Z}), *)$. A similar argument shows that the binary operation
\begin{equation}
x*y=x + \hat{i}(x,y) y \pmod n
\end{equation}
defines a quandle structure on $\mathrm{H}_1(S_g,\mathbb{Z}_n)$ containing $\mathcal{P}'_{g,n}$ as its subquandle. Further, for each $n \ge 2$, there is a surjective quandle homomorphism $\phi_n': \mathcal{P}'_g \to \mathcal{P}'_{g,n}$ given by $\phi_n'(x)=x\pmod n$, the reduction of $x$ modulo $n$.
\par
Consider the natural action of $\mathbb{Z}_2=\{1,-1\}$ on both $\mathcal{P}'_g$ and $\mathcal{P}'_{g,n}$. Let us set $\mathcal{P}_g:=\mathcal{P}'_g/\mathbb{Z}_2$ and $\mathcal{P}_{g,n}:=\mathcal{P}'_{g,n}/\mathbb{Z}_2$. It is clear that the same bilinear form $\hat{i}(-,-)$ gives quandle structures on $\mathcal{P}_g$ and $\mathcal{P}_{g,n}$. Further, $\phi_n'$ induces a surjective quandle homomorphism $\phi_n: \mathcal{P}_g \to \mathcal{P}_{g,n}$. We refer the quandles $\mathcal{P}_g$ and  $\mathcal{P}_{g,n}$ as {\it projective primitive homological quandles} of $S_g$.
\par

It is well-known \cite[Chapter 6]{Farb-Margalit2012} that there exists a surjective group homomorphism $$\Psi: \mathcal{M}_g\to \mathrm{Sp}(2g,\mathbb{Z})$$ induced by the action of the mapping classes on the first homology $\mathrm{H}_1(S_g, \mathbb{Z})$. For each $n \ge 2$, let  $$\Psi_n: \mathcal{M}_g\to \mathrm{Sp}(2g,\mathbb{Z}_n)$$ be the composition of $\Psi$ with the reduction modulo $n$ homomorphism. For the isotopy class $a$ of an oriented simple closed curve in $S_g$, we denote by $[a] \in \mathrm{H}_1(S_g, \mathbb{Z})$ its homology class. The main result of this section is the following.

\begin{theorem}\label{thm:homologicalquandle}
For each $g \ge 1$ and $n \ge 2$, there exist surjective quandle homomorphisms $$\mathcal{D}_g^{ns} \to \mathcal{P}_g$$ and $$\mathcal{D}_g^{ns} \to \mathcal{P}_{g, n}.$$
\end{theorem}

\begin{proof}
Recall that $\mathcal{D}_g^{ns}$ consists of isotopy classes of unoriented simple closed curves. Thus, for each $a \in \mathcal{D}_g^{ns}$, there are two choices for the homology class $[a] \in \mathcal{P}_g'$. We choose $[a]$ such that the entry in its first non-zero coordinate (from left) is positive. By \cite[Proposition 6.2]{Farb-Margalit2012}, a non-zero element of $\mathrm{H}_1(S_g, \mathbb{Z})$ is primitive  if and only if  it is represented by an oriented simple closed curve. Composing  the map $\mathcal{D}_g^{ns} \to \mathcal{P}'_g$ given by $a \mapsto [a]$ with the quotient map $\mathcal{P}'_g \to \mathcal{P}_g$, we get a surjective map $$\phi: \mathcal{D}_g^{ns} \to \mathcal{P}_g.$$ It remains to show that $\phi$ is a quandle homomorphism. If $a, b \in \mathcal{D}_g^{ns}$, then 
\begin{equation*}\label{symplectic-eq-1}
\phi(a*b)= \phi(T_b(a))=[T_b(a)]=\Psi(T_b)([a])
\end{equation*}
and
\begin{equation*}\label{symplectic-eq-2}
\phi(a)*\phi(b)=[a]*[b]=[a] +\hat{i}(a, b) [b].
\end{equation*}
By \cite[Proposition 6.3]{Farb-Margalit2012}, we have $\Psi(T_b)([a])=[a]+\hat{i}(a,b)[b]$, and hence $\phi$ is a quandle homomorphism. Finally, the composition $\phi_n\phi:\mathcal{D}_g^{ns}\to \mathcal{P}_{g,n}$ gives the desired second homomorphism.
\end{proof}

\begin{corollary}\label{cor:involutory_quandle}
For each $g \ge 1$, $\mathcal{D}_g^{ns}$ has an involutory quandle of size $2^{2g}-1$ as its quotient.
\end{corollary}

\begin{proof}
Since each element of $\mathrm{H}_1(S_g,\mathbb{Z}_2)$ other than the trivial element is primitive, the assertion follows from Theorem \ref{thm:homologicalquandle}.
\end{proof}

\begin{corollary}\label{cor:involutory_quandle-separating}
For each $g \ge 1$, $\mathcal{D}_g$ has an involutory quandle of size $2^{2g}$ as its quotient.
\end{corollary}

\begin{proof}
Since each separating simple closed curve represents the trivial element in homology, the result now follows from Theorem \ref{thm:homologicalquandle}.
\end{proof}
 
\begin{remark}
We make two remarks here.
\begin{enumerate}[(i)]
\item The definition of Dehn quandle of an orientable surface considered in \cite{Yetter2003} includes the isotopy class of trivial simple closed curves (non-essential curves). Mapping the isotopy class of trivial curves to a one element trivial quandle disjoint from $\mathrm{H}_1(S_g,\mathbb{Z}_2)$ and using Corollary \ref{cor:involutory_quandle-separating}, we obtain a generalisation of a similar result of Yetter \cite[Proposition 23]{Yetter2003} considering the genus two case.
\item The surjection in Corollary \ref{cor:involutory_quandle} in fact descends to a surjective quandle homomorphism from $(\mathcal{D}_g^{ns})_2$ onto an involutory quandle of size $2^{2g}-1$, where $(\mathcal{D}_g^{ns})_2$ is the canonical involutory quotient of $\mathcal{D}_g^{ns}$. 
\end{enumerate}
\end{remark}

\begin{lemma}\label{symp_property}
If $a, b\in \mathcal{D}_g^{ns}$ and $q$ is a prime, then the following hold:
	\begin{enumerate}[(i)]
		\item $\Psi(T_a)=\Psi(T_b)$ if and only if $[a]=[b]$  in $\mathcal{P}_g$. 
		\item $\Psi_q(T_a)=\Psi_q(T_b)$ if and only if $[a]=[b]$ in $\mathcal{P}_{g,q}$.
	\end{enumerate}
\end{lemma}

\begin{proof}
By \cite[Proposition 6.3]{Farb-Margalit2012}, we have $\Psi(T_a)[z]=[z]+\hat i(z,a)[a]$. Since $\hat{i}(z,a)$ depends only on homology classes of $z$ and $a$, it follows that if $[a]=[b]$, then $\Psi(T_a)=\Psi(T_b)$. For the converse, suppose that $\Psi(T_a)=\Psi(T_b)$. Then $[z]+\hat{i}(z,a)[a]=[z]+\hat{i}(z,b)[b]$, and hence $\hat{i}(z,a)[a]=\hat{i}(z,b)[b]$ for each simple closed curve $z$ in $S_g$. Choose a simple closed curve $z'$ such that $\hat{i}(z',b)=1$ (which always exists by the coordinate change principle). This gives $ [b]=\hat{i}(z',a)[a]$. Since $b$ is primitive, it follows that $\hat{i}(z',a)=\pm 1$,  and hence $[a]=[b]$ in $\mathcal{P}_g$, which proves assertion (i).
	\par
By definition, $\Psi_q(T_a)=\Psi(T_a)\pmod q$. Again, \cite[Proposition 6.3]{Farb-Margalit2012} gives  $\Psi_q(T_a)[z]=[z]+\hat{i}(z,a)[a] \pmod q$. If $[a]=[b]$ in $\mathcal{P}_{g,q}$, then $[a]=[b]+qk[v]$ in $\mathcal{P}_g$ for some integer $k$ and $[v] \in \mathrm{H}_1(S_g, \mathbb{Z})$.  Since $\hat{i}$ is bilinear, we have $\hat{i}(z,a)=\hat{i}(z,b)+qk\hat{i}(z,v)$. This implies that $\hat{i}(z,a)=\hat{i}(z,b)\pmod q$ for each simple closed curve $z$, and hence $\Psi_q(T_a)=\Psi_q(T_b)$. Conversely, suppose that $\Psi_q(T_a)=\Psi_q(T_b)$. This gives 
\begin{equation}\label{intersection-homological-equation}
\hat{i}(z,a)[a]=\hat{i}(z,b)[b] \pmod q
\end{equation}
in  $\mathcal{P}_{g,q}$ for each simple closed curve $z$ in $S_g$. Choosing a curve $z'$ such that $\hat{i}(z',b)=1$ and using \eqref{intersection-homological-equation}, we get 
 \begin{equation}\label{intersection-homological-equation3}
 [b]=\hat{i}(z',a)[a] \pmod q
 \end{equation}
 Since $[a]$ and $[b]$ are primitive elements, it follows that $\hat{i}(z',a)\in \mathbb{Z}_q^\times$. Substituting $[b]=\hat{i}(z',a)[a] \pmod q$ in \eqref{intersection-homological-equation}, we obtain 
\begin{equation}\label{intersection-homological-equation2}
\hat{i}(z,a)[a] =\hat{i}(z',a)^2 \hat{i}(z,a)[a]\pmod q
\end{equation}
 for each simple closed curve $z$. By choosing a simple closed curve $z''$ such that $\hat{i}(z'',a)=1$ and using \eqref{intersection-homological-equation2}, we get $ \hat{i}(z',a)^2=1 \pmod q$. Since $q$ is prime, $\mathbb{Z}_q^\times$ is cyclic, and hence $\hat{i}(z',a)=\pm 1$. Plugging this in \eqref{intersection-homological-equation3} gives $[a]=[b]$ in $\mathcal{P}_{g,q}$.
\end{proof}

We conclude by showing that projective primitive homological quandles are simply Dehn quandles of corresponding symplectic groups.

\begin{proposition}\label{homological-are-dehn-of-symplectic}
Let $a$ be a non-separating simple closed curve in $S_g$, where $g \ge 1$ and $q$ a prime. Then the following hold:
\begin{enumerate}[(i)]
\item $\mathcal{P}_{g}\cong\mathcal{D}(\Psi(T_a)^{\text{Sp}(2g,\mathbb{Z})})$.
\item $\mathcal{P}_{g,q}\cong\mathcal{D}(\Psi_q(T_a)^{\text{Sp}(2g,\mathbb{Z}_q)})$.
\end{enumerate}
\end{proposition}

\begin{proof}
It follows from \cite[Proposition 6.2]{Farb-Margalit2012} that each element of $\mathcal{P}_g$ is represented by a non-separating simple closed curve in $S_g$ (without orientation). Define $\lambda:\mathcal{P}_g \to \Conj(\text{Sp}(2g,\mathbb{Z}))$ by $\lambda([b])=\Psi(T_b)$ for each $[b] \in \mathcal{P}_{g}$. It follows from Lemma \ref{symp_property}(i) that $\lambda$ is well-defined and injective. Further, 
$$\lambda([a]*[b])=\lambda([T_b(a)])=\Psi(T_{T_b(a)})=\Psi(T_bT_aT_b^{-1})=\Psi(T_b)\Psi(T_a)\Psi(T_b)^{-1}$$
and
$$\lambda([a])*\lambda([b])=\Psi(T_a)*\Psi(T_b)=\Psi(T_b)\Psi(T_a)\Psi(T_b)^{-1},$$ 
and hence $\mathcal{P}_g\cong\{\Psi(T_b)\mid b\in \mathcal{D}_g^{ns}\}\subset \Conj(\text{Sp}(2g,\mathbb{Z}))$. Since, Dehn twists along any two non-separating  simple closed curves are conjugate in $\mathcal{M}_g$, their images under $\Psi$ are also conjugate in $\text{Sp}(2g,\mathbb{Z})$. Further, since $\mathcal{M}_g$ is generated by $\mathcal{D}_g^{ns}$, it follows that $\text{Sp}(2g,\mathbb{Z})$ is generated by $\{\Psi(T_b) \mid b\in \mathcal{D}_g^{ns}\}$. Hence, $\mathcal{P}_{g}\cong\mathcal{D}(\Psi(T_a)^{\text{Sp}(2g,\mathbb{Z})})$ for a non-separating simple closed curve $a$, which is assertion (i). 
\par	

Define $\lambda_q:\mathcal{P}_{g,q} \to \Conj(\text{Sp}(2g,\mathbb{Z}_q))$ by $\lambda_q([b]_q)=\Psi_q(T_b)$ for  each $[b]_q \in \mathcal{P}_{g, q}$. By Lemma \ref{symp_property}(ii), $\lambda_p$ is well-defined and injective. Further, we have $$\lambda_q([a]_q*[b]_q)=\lambda_q([T_b(a)]_q)=\Psi_q(T_{T_b(a)})=\Psi_q(T_bT_aT_b^{-1})=\Psi_q(T_b)\Psi_q(T_a)\Psi_q(T_b)^{-1}$$
and
$$\lambda_q([a]_q)*\lambda_q([b]_q)=\Psi_q(T_a)*\Psi_q(T_b)=\Psi_q(T_b)\Psi_q(T_a)\Psi_q(T_b)^{-1}.$$ 
Thus, $\mathcal{P}_{g,q}\cong\{\Psi_q(T_b) \mid b\in \mathcal{D}_g^{ns}\}\subset \Conj(\text{Sp}(2g,\mathbb{Z}_q))$. Since, Dehn twists along two non-separating simple closed curves are conjugate in $\mathcal{M}_g$,  their images under $\Psi_q$ are also conjugate in $\text{Sp}(2g,\mathbb{Z}_q)$. Further, $\text{Sp}(2g,\mathbb{Z}_q)$ is generated by $\{\Psi_q(T_b) \mid b\in \mathcal{D}_g^{ns}\}$ as $\mathcal{M}_g$ is generated by $\mathcal{D}_g^{ns}$. Thus, we get $\mathcal{P}_{g,q}\cong\mathcal{D}(\Psi_q(T_a)^{\text{Sp}(2g,\mathbb{Z}_q)})$ for a non-separating simple closed curve $a$, which proves assertion (ii).
\end{proof}
\medskip

\begin{ack}
Neeraj K. Dhanwani thanks IISER Mohali for the institute post doctoral fellowship. Hitesh Raundal is supported by research associateship under the SERB grant SB/ SJF/2019-20. Mahender Singh is supported by the Swarna Jayanti Fellowship grants DST/SJF/MSA-02/2018-19 and SB/SJF/2019-20.
\end{ack}


\begin{thebibliography}{HD}

\bibitem{TAkita} T. Akita, \textit{The adjoint group of a Coxeter quandle}, Kyoto J. Math. 60 (2020), 1245--1260.

\bibitem{Andruskiewitsch2003} N. Andruskiewitsch and M. Gra\~{n}a, \textit{From racks to pointed Hopf algebras}, Adv. Math. 178 (2003), no. 2, 177--243.

\bibitem{BardakovNasybullov2020} V. Bardakov and T. Nasybullov, \textit{Embeddings of quandles into groups}, J. Algebra Appl. 19 (2020), no. 7, 2050136, 20 pp.

\bibitem{BardakovNasybullovSingh2019} V. G. Bardakov, T. R. Nasybullov and M. Singh, \textit{Automorphism groups of quandles and related groups}, Monatsh. Math. 189 (2019), 1--21.

\bibitem{MR4450681} V. G. Bardakov, I. B. S. Passi and M. Singh, \textit{Zero-divisors and idempotents in quandle rings}, Osaka J. Math. 59 (2022), no. 3, 611--637. 


\bibitem{Bergman2021}  G. M. Bergman,  \textit{On core quandles of groups}, Comm. Algebra 49 (2021) 2516--2537.

\bibitem{Berrick1989} A. J. Berrick, \textit{Universal groups, binate groups and acyclicity}, Group theory, Proc. Conf., Singapore 1987, 253--266 (1989).

\bibitem{Brown1982} K. S. Brown, \textit{Cohomology of Groups}, Graduate Texts in Mathematics, Vol. 87 (Springer-Verlag, New York, 1982), x+306 pp.

\bibitem{ChamanaraHuZablow} R. Chamanara, J. Hu and J. Zablow, \textit{Extending the Dehn quandle to shears and foliations on the torus}, Fund. Math. 225 (2014), no. 1, 1--22.

\bibitem{Davis2008} M. W. Davis, \textit{The geometry and topology of Coxeter groups}, London Mathematical Society Monographs Series, 32. Princeton University Press, Princeton, NJ, 2008. xvi+584 pp. 

\bibitem{DhanwaniRaundalSingh2022} N. K. Dhanwani, H. Raundal and M. Singh \textit{Presentations of Dehn quandles},  (2022), arXiv:2202.02531v2.


\bibitem{Eisermann2005} M. Eisermann, \textit{Yang-Baxter deformations of quandles and racks}, Algebr. Geom. Topol. 5 (2005), 537--562.

\bibitem{Farb-Margalit2012} B. Farb and D. Margalit, \textit{A primer on mapping class groups}, Princeton Mathematical Series, 49. Princeton University Press, Princeton, NJ, 2012. xiv+472 pp.

\bibitem{FennRourke1997} R. Fenn, R. Rimanyi and C. Rourke, \textit{The braid-permutation group}, Topology 36 (1997), no. 1, 123--135. 


\bibitem{Gervais1996} S. Gervais, \textit{Presentation and central extensions of mapping class groups}, Trans. Amer. Math. Soc. 348 (1996), no. 8, 3097-3132.

\bibitem{HarpeMcDuff} P. de la Harpe and D. McDuff, \textit{Acyclic groups of automorphisms}, Comment. Math. Helv. 58 (1983), 48--71.


\bibitem{Ivanov}  N. V. Ivanov, \textit{Automorphisms of Teichm\"{u}ller modular groups}, Topology and geometry--Rohlin Seminar, 199--70, Lecture Notes in Math., 1346, Springer, Berlin, 1988.

\bibitem{Joyce1979} D. Joyce, \textit{An algebraic approach to symmetry with applications to knot theory}, Thesis (Ph.D.)--University of Pennsylvania. 1979. 127 pp.

\bibitem{Joyce1982} D. Joyce, \textit{A classifying invariant of knots, the knot quandle}, J. Pure Appl. Algebra 23 (1982), 37--65.

\bibitem{Kamada2012} S. Kamada, \textit{Kyokumen musubime riron (Surface-knot theory)}, (in Japanese), Springer Gendai Sugaku Series 16 (2012), Maruzen Publishing Co. Ltd.

\bibitem{Kamada2017} S. Kamada, \textit{Surface-knots in 4-space. An introduction}, Springer Monographs in Mathematics. Springer, Singapore, 2017. xi+212 pp.

\bibitem{kamadamatsumoto} S. Kamada and Y. Matsumoto, \textit{Certain racks associated with the braid groups}, Knots in Hellas '98 (Delphi), 118--130, Ser. Knots Everything, 24, World Sci. Publ., River Edge, NJ, 2000.

\bibitem{Kassel-Turaev2008} C. Kassel and V. Turaev, \textit{Braid Groups}, With the graphical assistance of Olivier Dodane. Graduate Texts in Mathematics, 247. Springer, New York, 2008. xii+340 pp.

\bibitem{Labruere-Paris2001} C. Labruere and L. Paris, \textit{Presentations for the punctured mapping class in terms of Artin groups}, Algebr. Geom. Topol. 1 (2001), 73--114.

\bibitem{Loos1} O. Loos, \textit{Reflexion spaces and homogeneous symmetric spaces}, Bull. Amer. Math. Soc. 73 (1967), 250--253.

\bibitem{Matveev1982} S. V. Matveev, \textit{Distributive groupoids in knot theory}, in Russian: Mat. Sb. (N.S.) 119 (1) (1982) 78--88, translated in Math. USSR Sb. 47 (1) (1984), 73--83.

\bibitem{MulhollandRolfsen} J. Mulholland and D. Rolfsen, \textit{Local indicability and commutator subgroups of Artin groups}, (2006), arXiv:math/0606116.

\bibitem{NiebrzydowskiPrzytycki2009} M. Niebrzydowski and J. H. Przytycki, \textit{The quandle of the trefoil knot as the Dehn quandle of the torus},  Osaka J. Math. 46 (2009), 645--659.


\bibitem{MR4028088} T. Nosaka, \textit{Finite presentations of centrally extended mapping class groups}, Kyushu J. Math. 73 (2019), no. 1, 103--113. 

\bibitem{Nosaka2017} T. Nosaka, \textit{Central extensions of groups and adjoint groups of quandles}, Geometry and analysis of discrete groups and hyperbolic spaces, 167--184, RIMS Kokyuroku Bessatsu, B66, Res. Inst. Math. Sci. (RIMS), Kyoto, 2017.

\bibitem{NosakaBook} T. Nosaka, \textit{Quandles and topological pairs. Symmetry, knots, and cohomology}, SpringerBriefs in Mathematics. Springer, Singapore, 2017. ix+136 pp.


\bibitem{RafiSchleimer2011} K. Rafi and S. Schleimer, \textit{Curve complexes are rigid}, Duke Math. J. 158 (2011), no. 2, 225--246.


\bibitem{RaundalSinghSingh2020} H. Raundal, M. Singh and M. Singh, \textit{Orderability of link quandles}, Proc. Edinb. Math. Soc. (2) 64 (2021), no. 3, 620--649. 

\bibitem{Ryder1996} H. Ryder, \textit{An algebraic condition to determine whether a knot is prime}, Math. Proc. Cambridge Philos. Soc. 120 (1996), no. 3, 385--389.

\bibitem{Vinogradov1949} A. A. Vinogradov, \textit{On the free product of ordered groups}, (Russian) Mat. Sbornik N.S. 25(67) (1949), 163--168.

\bibitem{Winker1984} S. K. Winker, \textit{Quandles, knots invariants and the n-fold branched cover}, Ph.D. Thesis, University of Illinois at Chicago, 1984.

\bibitem{Yetter2002} D. N. Yetter, \textit{Quandles and Lefschetz fibrations}, (2002), arXiv:math/0201270.

\bibitem{Yetter2003} D. N. Yetter, \textit{Quandles and monodromy}, J. Knot Theory Ramifications 12 (2003), no. 4, 523--541.

\bibitem{Zablow1999} J. Zablow, \textit{Loops, waves, and an ``algebra'' for Heegaard splittings}, Thesis (Ph.D.)-City University of New York. 1999. 64 pp.

\bibitem{Zablow2003} J. Zablow, \textit{Loops and disks in surfaces and handlebodies}, J. Knot Theory Ramifications 12 (2003), no. 2, 203--223.

\bibitem{Zablow2008} J. Zablow, \textit{On relations and homology of the Dehn quandle}, Algebr. Geom. Topol. 8 (2008), no. 1, 19--51.

\bibitem{Zablow2014} J. Zablow, \textit{Links and cycles in Dehn quandle homology I: Torus links}, J. Knot Theory Ramifications 23 (2014), no. 2, 1450012, 24 pp.

\end{thebibliography}
\end{document}